\documentclass[12pt]{amsart}
\usepackage{amssymb}
\usepackage{amsfonts}
\usepackage{amssymb,latexsym}
\usepackage{enumerate}
\usepackage{mathrsfs}
\usepackage{url}
\allowdisplaybreaks

\makeatletter
\@namedef{subjclassname@2020}{%
	\textup{2020} Mathematics Subject Classification}
\makeatother

[2002/01/19 v2.2g %
AMS font definitions%
]
\DeclareFontFamily{U}{euf}{}
\DeclareFontShape{U}{euf}{m}{n}{%
	<5><6><7><8><9>gen*eufm%
	<10><10.95><12><14.4><17.28><20.74><24.88>eufm10%
}{}
\DeclareFontShape{U}{euf}{b}{n}{%
	<5><6><7><8><9>gen*eufb%
	<10><10.95><12><14.4><17.28><20.74><24.88>eufb10%
}{}

[2002/01/19 v2.2g %
AMS font definitions%
]
\DeclareFontFamily{U}{msb}{}
\DeclareFontShape{U}{msb}{m}{n}{%
	<5><6><7><8><9>gen*msbm%
	<10><10.95><12><14.4><17.28><20.74><24.88>msbm10%
}{}

[2002/01/19 v2.2g %
AMS font definitions%
]
\DeclareFontFamily{U}{msa}{}
\DeclareFontShape{U}{msa}{m}{n}{%
	<5><6><7><8><9>gen*msam%
	<10><10.95><12><14.4><17.28><20.74><24.88>msam10%
}{}

\newtheorem{theorem}{Theorem}[section]
\newtheorem{lemma}[theorem]{Lemma}
\newtheorem{proposition}[theorem]{Proposition}

\theoremstyle{definition}

\newtheorem{definition}[theorem]{Definition}
\newtheorem{Example}[theorem]{Example}

\numberwithin{equation}{section} 

\begin{document}
	
	\def\Re{{\rm Re}}
	\def\Im{{\rm Im}}
	\def\Li{{\rm Li}}

	\title[On the properties of alternating invariant functions]
	{On the properties of alternating invariant functions}

	\begin{abstract}
		Functions satisfying the functional equation  
		\begin{align*}
			\sum_{r=0}^{n-1} (-1)^r f(x+ry, ny) = f(x,y), \quad \text{for any positive odd integer $n$},
					\end{align*} 
		are named the alternating invariant functions. 
		Examples of such functions include Euler polynomials, alternating Hurwitz zeta functions and their associated Gamma functions. 
		In this paper, we systematically investigate the fundamental properties of alternating invariant functions. We prove that the set of such functions is closed under translation, reflection, and differentiation. In addition, we define a convolution operation on alternating invariant functions and derive explicit convolution formulas for Euler polynomials and alternating Hurwitz zeta functions, respectively. Furthermore, using distributional relations, we construct new examples of alternating invariant functions, including suitable combinations of trigonometric, exponential, and logarithmic functions, among others.	\end{abstract}
	
	\author{Haiqing Zhu}
	\address{Department of Mathematics, South China University of Technology, Guangzhou, Guangdong 510640, China}
	\email{1094836815@qq.com}
	\author{Su Hu}
	\address{Department of Mathematics, South China University of Technology, Guangzhou, Guangdong 510640, China}
	\email{mahusu@scut.edu.cn}
	\author{Min-Soo Kim}
	\address{Department of Mathematics Education, Kyungnam University, Changwon, Gyeongnam 51767, Republic of Korea}
	\email{mskim@kyungnam.ac.kr}


	\maketitle
	
	\vspace*{1ex}

	\section{Introduction}
	\subsection{Bernoulli polynomial and invariant functions} 
	The Bernoulli polynomials \( B_n(x) \) (\( n \in \mathbb{N}_0 = \mathbb{N} \cup \{0\} \)) are defined by the generating function
	\begin{align*}
		\frac{te^{xt}}{e^t - 1} = \sum_{k=0}^{\infty} B_k(x) \frac{t^k}{k!}.
	\end{align*}
	They possess many interesting properties, one of which is the following distribution formula (see \cite[Theorem 3.1]{2}):
	\begin{align*}
		\sum_{r=0}^{n - 1} B_m\left(x + \frac{r}{n}\right) = n^{1 - m} B_m(nx), \quad\textrm{for}~ m \in \mathbb{N}_0 \ \textrm{and}~  n \in \mathbb{N}.
	\end{align*}
	Let \( f(x, y) = y^{m - 1} B_m\left(\frac{x}{y}\right)\), based on the above equation, we introduce the following definition: 
	\begin{definition}[{See \cite[Eq. (1.2)]{33}}]\label{def:invariant function}
A real function $f(x,y)$ defined for $y>0$ is called an \textit{invariant function} if it satisfies the functional equation
    \begin{equation}\label{(1-1)}
		\begin{aligned}
			\sum_{r=0}^{n - 1} f(x + r y, n y) = f(x, y), \quad\textrm{for any}~ n \in \mathbb{N}.
		\end{aligned}
	\end{equation}
    The set of all invariant functions is denoted by $\textup{InV}$.
\end{definition}
This class of functions was defined by Z.-W. Sun in \cite{30,31,32} and Z.-H. Sun in \cite{33}.
Besides Bernoulli polynomials,  Z.-H. Sun also provided many examples of invariant functions, such as the Hurwitz zeta function and the Gamma function (see \cite[pp. 7–11]{33}):
	\begin{Example}[{See \cite[Example 2.2]{33}}]
		The Hurwitz zeta function is defined as follows (see \cite[Definition 9.6.1]{19}):
		\begin{align*}
			\zeta(s, x) = \sum_{n=0}^{\infty} \frac{1}{(n + x)^s},
		\end{align*}
		where \( \textup{Re}(s) > 1 \) and \( x \in \mathbb{C} \setminus \mathbb{Z}_{\leq 0} = \{0, -1, -2, \ldots\} \). It can be analytically continued to a meromorphic function on the complex plane with a simple pole at \( s = 1 \).  
		For \( \textup{Re}(s) < 0 \), the distribution formula of the Hurwitz zeta function  gives (see \cite[Proposition 9.6.12]{19}):
		\begin{align*}
			\sum_{0 \leq j < N} \zeta\left(s, x + \frac{j}{N}\right) = N^s \zeta(s, Nx),
		\end{align*}
		implying \( y^{-s} \zeta\left(s, \frac{x}{y}\right) \in \textup{InV} \). Similarly, for \( \textup{Re}(s) > 1 \), \( y > 0 \), and \( \frac{x}{y} \in \mathbb{C} \setminus \mathbb{Z}_{\leq 0} \), we have \( y^{-s} \zeta\left(s, \frac{x}{y}\right) \in \textup{InV} \).
	\end{Example}
	
	\begin{Example}[{See \cite[Example 2.12]{33}}]
		Euler's Gamma function is defined as follows:
		\begin{align*}
			\Gamma(x) = \int_0^{\infty} t^{x-1} e^{-t} \, dt, \quad \text{where } \textup{Re}(x) > 0.
		\end{align*}
		For \( x \in \mathbb{R} \) and \( y > 0 \), define the function \( f(x, y) \) as:
		\begin{align*}
			f(x, y) = 
			\begin{cases} 
				\log \left| \frac{y^{\frac{x}{y}} \Gamma\left(\frac{x}{y}\right)}{\sqrt{2\pi y}} \right| & \text{if } \frac{x}{y} \notin \mathbb{Z}_{\leq 0}, \\
				\log \frac{y^{\frac{x}{y}}}{\left(-\frac{x}{y}\right)! \sqrt{2\pi y}} & \text{if } \frac{x}{y} \in \mathbb{Z}_{\leq 0}.
			\end{cases}
		\end{align*}
		By the distribution formula of the Gamma function (see \cite[p. 99, Eq. (1)]{43}):
		\begin{align*}
			\prod_{0 \leq j < N} \Gamma\left(s + \frac{j}{N}\right) = N^{1/2 - Ns} (2\pi)^{(N-1)/2} \Gamma(Ns),
		\end{align*}
		we have \( f \in \textup{InV} \).
	\end{Example}
	
	
	\begin{Example}[{See \cite[Example 2.11]{33}}]  
		For \( y > 0 \), define
		\begin{align*}
			f(x, y) = 
			\begin{cases} 
				\frac{1}{y} \cot \pi \frac{x}{y} & \text{if } \frac{x}{y} \notin \mathbb{Z}, \\
				0 & \text{if } \frac{x}{y} \in \mathbb{Z}.
			\end{cases}
		\end{align*}
		From the partial fraction expansion of \( \pi \cot \pi v \) (see \cite[p. 9]{32}):
		\begin{align*}
			\pi \cot \pi v = \sum_{k=-\infty}^{+\infty} \frac{1}{v + k} = \frac{1}{v} + 2v \sum_{k=1}^{\infty} \frac{1}{v^2 - k^2},
		\end{align*}
		we obtain \( f \in \textup{InV} \).
	\end{Example}
	
	Sun defined an operation \( \ast \) on the set \( \textup{InV} \):
	\begin{equation}\begin{aligned}
			g \ast h(x, y) = \int_{0}^{x} g(t, y) h(x - t, y) \, dt + \int_{x}^{y} g(t, y) h(x + y - t, y) \, dt,
	\end{aligned}\end{equation}
	and proved that \( g \ast h(x, y) \in \textup{InV} \). In fact, he established the following theorem:
	\begin{theorem}[{See \cite[Theorem 3.4]{33}}]\label{Theorem 1.1}  
		If \( g(x, y), h(x, y) \in \textup{InV} \), then \( g \ast h(x, y) \in \textup{InV} \) and
		\begin{align*}
			\int_{0}^{y} g \ast h(x, y) \, dx = \int_{0}^{y} g(x, y) \, dx \cdot \int_{0}^{y} h(x, y) \, dx.
		\end{align*}
	\end{theorem}
	He also provided convolution formulas for Bernoulli polynomials and the Hurwitz zeta function, as follows:
	\begin{theorem}[{See \cite[Theorem 3.6]{33}}]\label{Theorem 1.2}
		For \( m, n \in \mathbb{N} \), we have
		\begin{align*}
			B_{m+n}(x) = -\binom{m+n}{m} \left( \int_{0}^{1} B_m(x - t) B_n(t) \, dt + m \int_{x}^{1} (x - t)^{m-1} B_n(t) \, dt \right)
		\end{align*}
		and
		\begin{align*}
			\frac{-y^{m-1} B_m(\frac{x}{y})}{m!} \ast \frac{-y^{n-1} B_n(\frac{x}{y})}{n!} = \frac{-y^{m+n-1} B_{m+n}(\frac{x}{y})}{(m+n)!}.
		\end{align*}
	\end{theorem}
	\begin{theorem}[{See \cite[Eq. (3.9)]{33}}]\label{Theorem 1.3}  
		For \( \alpha, \beta > 1 \), we have
		\begin{align*}
			\frac{y^{\alpha-1} \zeta(1 - \alpha, \frac{x}{y})}{\Gamma(\alpha)} \ast \frac{y^{\beta-1} \zeta(1 - \beta, \frac{x}{y})}{\Gamma(\beta)} = \frac{y^{\alpha+\beta-1} \zeta(1 - \alpha - \beta, \frac{x}{y})}{\Gamma(\alpha + \beta)}.
		\end{align*}
	\end{theorem}
	Furthermore, Sun  proved many other properties of invariant functions under the differentiation and the integration (see \cite[Theorems 3.1–3.3]{33}). In particular, \cite[Theorem 3.1(iii)]{33} shows that if \( \frac{\partial f}{\partial y} \) exists, then
	\begin{align*}
		f(x, y) = \int_{x}^{x - y} \frac{\partial}{\partial y} f(t, y) \, dt.
	\end{align*}
	
	\subsection{Euler polynomial and alternating invariant functions}
	As a ‘dual’ to the Bernoulli polynomials, the Euler polynomials $E_n(x)$ are defined by the following generating function (see \cite[Theorem 2]{2}):
	\begin{align*}
		\frac{2e^{xt}}{e^t+1} = \sum\limits_{k=0}^{\infty} E_k(x) \frac{t^k}{k!}.
	\end{align*}
	The Euler polynomials were introduced by Euler in 1755 while studying alternating power sums. He proved that for $m, n \in \mathbb{N}_0$, the following holds (see \cite[p. 932]{10}):
	\begin{align*}
		\sum_{j=0}^n (-1)^j j^m = \frac{(-1)^n E_m(n + 1) + E_m(0)}{2}.
	\end{align*}
	
	If $n$ is a positive odd integer, then we have the following distribution formula (see \cite[Theorem 3.3]{2}):
	\begin{equation}\label{(1-3)}
		\begin{aligned}
			\sum_{r=0}^{n-1} (-1)^r E_m \left( \frac{x + r}{n} \right) = \frac{E_m(x)}{n^m}, \quad 
			m \in \mathbb{N}_0.
	\end{aligned}\end{equation}
	
	The alternating Hurwitz zeta function, as a ‘dual’ form of the Hurwitz zeta function, is defined by the following series (see \cite[Eq. 1.5]{15}):
	\begin{equation}\label{(1-4)}
		\begin{aligned}   
			\zeta_E(s, x) = \sum_{n=0}^\infty \frac{(-1)^n}{(n + x)^s}, \quad \text{Re}(s) > 0.
	\end{aligned}\end{equation}
	It can be analytically continued to the entire complex plane $\mathbb{C}$ and has no poles. We have (see \cite[Eq. (1.15)]{14}, \cite[Eq. (1.10)]{17}):
	\begin{align*}
		\zeta(1-m, x) = -\frac{B_{m}(x)}{m}, \quad m \in \mathbb{N}
	\end{align*}
	and
	\begin{align*}
		\zeta_E(-m, x) = \frac{1}{2} E_m(x), \quad m \in \mathbb{N}_{0}.
	\end{align*}
	
	Recently, the alternating Hurwitz zeta function $\zeta_E(s, x)$ has been systematically studied, including its Fourier expansions, power series expansions, special values, integral representations, and convexity properties (see, e.g., \cite{22,15,23,17}). Moreover, in algebraic number theory, the alternating Hurwitz zeta function $\zeta_E(s, x)$ can be used to represent partial zeta functions of cyclotomic fields in Stark's conjecture (see \cite[Eq. (6.13)]{44}). In particular, when $x = 1$, the function $\zeta_E(s, x)$ reduces to the alternating zeta function (also known as the Dirichlet eta function, see \cite[p. 952]{23}):
	\begin{equation}\label{(1-5)}
		\begin{aligned}
			\zeta_E(s, 1) = \zeta_E(s) = \sum_{n=1}^{\infty} \frac{(-1)^{n+1}}{n^s} = \eta(s),
		\end{aligned}
	\end{equation}
	and we also have
	\begin{align*}
		\zeta_E(s) = \left(1 - \frac{1}{2^{s-1}}\right) \zeta(s).
	\end{align*}
	The following difference equation can be directly derived from (\ref{(1-4)}):
	\begin{align*}
		\zeta_E(s, x + 1) + \zeta_E(s, x) = x^{-s}.
	\end{align*}
	This further leads to the distribution formula for $\zeta_E(s, x)$ (see \cite[p. 2987]{46}):
	\begin{equation}\label{(1-6)}
		\begin{aligned}
			\zeta_E(s, n x) = n^{-s} \sum_{j=0}^{n-1} (-1)^j \zeta_E\left(s, x + \frac{j}{n}\right).
	\end{aligned}\end{equation}
	
	In 2021, Hu and Kim defined the generalized Stieltjes constant $\tilde{\gamma}_k(x)$ and the Euler constant $\tilde{\gamma}_0$ via the Taylor expansion of $\zeta_E(s, x)$ at $s = 1$ (see \cite[Eq. (1.18)]{16}):
	$$
	\zeta_E(s, x) = \sum_{k=0}^{\infty} \frac{(-1)^k \tilde{\gamma}_k(x)}{k!} (s - 1)^k.
	$$
	From this expansion, we obtain
	\begin{equation}\label{(1-7)}
		\begin{aligned}
			\tilde{\gamma}_0(x) = \zeta_E(1, x),
	\end{aligned}\end{equation}
	and denote
	\begin{equation}\label{(1-8)}
		\begin{aligned}
			\tilde{\gamma}_k := \tilde{\gamma}_k(1).
	\end{aligned}\end{equation}
	For $k \in \mathbb{N}_0$, and based on Eqs. (\ref{(1-5)}), (\ref{(1-7)}), and (\ref{(1-8)}), we have
	$$
	\tilde{\gamma}_0 = \tilde{\gamma}_0(1) = \zeta_E(1) = \sum_{n=1}^{\infty} \frac{(-1)^{n+1}}{n} = \log 2
	$$
	(see \cite[Eqs. (1.17) and (1.20)]{16}).
	They also defined the corresponding digamma function $\tilde{\psi}(x)$ as follows:
	$$
	\tilde{\psi}(x) := -\tilde{\gamma}_0(x) = -\zeta_E(1, x)
	$$
	(see \cite[Eq. (1.22)]{16}).
	The corresponding Gamma function $\tilde{\Gamma}(x)$ is given by the following differential equation (see \cite[p. 5]{16}):
	$$
	\tilde{\psi}(x) = \frac{\tilde{\Gamma}'(x)}{\tilde{\Gamma}(x)} = \frac{d}{dx} \log \tilde{\Gamma}(x),
	$$
	and has the following Weierstrass-type infinite product expansion (see \cite[Theorem 3.12]{16}):
	$$
	\tilde{\Gamma}(x) = \frac{1}{x} e^{\tilde{\gamma}_0 x} \prod_{k=1}^{\infty} \left( e^{-\frac{x}{k}} \left(1 + \frac{x}{k}\right) \right)^{(-1)^{k+1}}.
	$$
	
	Recently, Wang, Hu, and Kim \cite{47} further investigated the properties of the Gamma function $\tilde{\Gamma}(x)$, obtaining several results in analogy with those for the classical Gamma function $\Gamma(x)$, such as the integral representation, the recurrence relation, the generalized logarithmic convexity, the reflection formula, Legendre's duplication formula, the special values and Lerch's formula. In particular, we have the following distribution formula (see \cite[Theorem 2.17]{47}):
	\begin{equation}\label{(1-9)}
		\begin{aligned}
			\tilde{\Gamma}(n x) = \frac{1}{\sqrt{n}} \prod_{j=0}^{n-1} \tilde{\Gamma}\left(x + \frac{j}{n}\right)^{(-1)^j},
	\end{aligned}\end{equation}
	where $n$ is any positive odd integer.
	
	Inspired by Eqs. (\ref{(1-3)}), (\ref{(1-6)}), and (\ref{(1-9)}), we define a new class of invariant functions by the following relation:
\begin{definition}\label{stardef}
    A real function $f(x,y)$ defined for $y>0$ is called an \textit{alternating invariant function} if it satisfies the functional equation
  \begin{equation}\label{star}
		\begin{aligned}
			\sum\limits_{r=0}^{n-1} (-1)^r f(x + r y, n y) = f(x, y), \quad \text{for any positive odd integer $n$}.
		\end{aligned}
	\end{equation}
    The set of all alternating invariant functions is denoted by $\mathrm{InV}^*$.
\end{definition}	
	In the following, we will show that from Eq. (\ref{(1-3)}), we obtain $f(x, y) = y^m E_m\left(\frac{x}{y}\right) \in \mathrm{InV}^*$ (see Example \ref{Example 5.1}); from Eq. (\ref{(1-6)}), for $\text{Re}(s) < 0$ and $y > 0$, we have $y^{-s} \zeta_E\left(s, \frac{x}{y}\right) \in \mathrm{InV}^*$, and for $\text{Re}(s) > 0$, $y > 0$, and $\frac{x}{y} \notin \mathbb{Z}_{\leq 0}$, we have $y^{-s} \zeta_E\left(s, \frac{x}{y}\right) \in \mathrm{InV}^*$ (see Example \ref{Example 5.2}); from Eq. (\ref{(1-9)}), for $\frac{x}{y} \notin \mathbb{Z}_{\leq 0}$, we have $\log \left| \frac{\tilde{\Gamma}\left(\frac{x}{y}\right)}{\sqrt{y}} \right| \in \mathrm{InV}^*$ (see Example \ref{Example 5.3}).
	
	The following theorem introduces an operation $*$ on the set $\mathrm{InV}^*$.
	\begin{theorem}\label{Theorem 1.4}
		For $g(x,y)$ and $h(x, y) \in \mathrm{InV}^*$, define
		\begin{equation}\label{(1-10)}
			\begin{aligned}
				g \ast h(x, y) = \int_{0}^{x} g(t, y) h(x - t, y) \, dt - \int_{x}^{y} g(t, y) h(x + y - t, y) \, dt,
		\end{aligned}\end{equation}
		then $g \ast h(x, y) \in \mathrm{InV}^*$.
	\end{theorem}
	\begin{Example}\label{Example 5.1}
		$y^mE_m\left(\frac{x}{y}\right) \in \mathrm{InV}^*\, (m \in \mathbb{N}).$
	\end{Example}
	\begin{proof}
		For $y > 0$ and $n$ being an odd number, we have
		$$\sum\limits_{r=0}^{n-1}(-1)^r(ny)^mE_m\left(\frac{x+ry}{ny}\right)=\sum\limits_{r=0}^{n-1}(-1)^r(ny)^mE_m\left( \frac{\frac{x}{y}+r}{n}\right).$$
		Then from the distribution equation of Euler polynomials (\ref{(1-3)}), we obtain
		\begin{align*}
			\sum\limits_{r=0}^{n-1}(-1)^r(ny)^mE_m\left( \frac{x+ry}{ny}\right)
			&=\frac{E_m\left( \frac{x}{y}\right) }{n^m}\cdot n^my^m\\
			&=y^mE_m\left( \frac{x}{y}\right).
		\end{align*}
		Thus, $y^mE_m\left( \frac{x}{y}\right) \in \mathrm{InV}^*.$
		This completes the proof.
	\end{proof} 
	In analogy with Theorems \ref{Theorem 1.2}, we have the following convolution formulas for Euler polynomials.
	\begin{theorem}\label{Theorem 1.5}
		For any $m, n \in \mathbb{N}$, we have
		\begin{align*}
			E_{m+n+1}(x) = \frac{\binom{m+n}{m} (m+n+1)}{2} \left( \int_{0}^{1} E_m(x - t) E_n(t) \, dt - 2 \int_{x}^{1} (x - t)^m E_n(t) \, dt \right)
		\end{align*}
		and
		\begin{align*}
			\frac{y^m E_m\left(\frac{x}{y}\right)}{2 m!} \ast \frac{y^n E_n\left(\frac{x}{y}\right)}{2 n!} = \frac{y^{m+n+1} E_{m+n+1}\left(\frac{x}{y}\right)}{2 (m+n+1)!}.
		\end{align*}
	\end{theorem}
	\begin{Example}\label{Example 5.2}
		For $\textup{Re}(s)<0$ and $y>0$, we have
		$y^{-s}\zeta_E\left( s,\frac{x}{y}\right) \in \mathrm{InV}^*$;
		and for $\textup{Re}(s)>0$, $y>0$ with $\frac{x}{y}\in\mathbb{C}\setminus\mathbb{Z}_{\leq0}$, we have
		$y^{-s}\zeta_E\left( s,\frac{x}{y}\right) \in \mathrm{InV}^*$.
	\end{Example}
	
	\begin{proof}
		By \cite[Eq. (1.7)]{20}, for $\textup{Re}(s)<1$ and $x\in(0,1]$, we have the Fourier expansion:
		\begin{equation}\label{zetaFour}
			\zeta_E(s,x)=\frac{2\Gamma(1-s)}{\pi^{1-s}}\left[\sin \frac{\pi s}{2}\sum_{m=0}^{\infty}\frac{\cos\left( (2m+1)\pi x\right)}{(2m+1)^{1-s}}+\cos \frac{\pi s}{2}\sum_{m=0}^{\infty}\frac{\sin\left( (2m+1)\pi x\right)}{(2m+1)^{1-s}}\right].
		\end{equation}
		Now consider the case $\textup{Re}(s)<0$ and $\frac{x}{y}\in(0,1]$. From the above expression, we obtain
		\begin{align*}
			&\quad\frac{2\Gamma(1-s)}{\pi^{1-s}}\sum_{r=0}^{n-1}(-1)^r\frac{1}{ny}\sin \frac{\pi s}{2}\sum_{m=0}^{\infty}\left(\frac{2m+1}{ny}\right)^{s-1}\cos\left( (2m+1)\pi \frac{x+ry}{ny}\right) \\
			&\quad+ \frac{2\Gamma(1-s)}{\pi^{1-s}}\sum_{r=0}^{n-1}(-1)^r\frac{1}{ny}\cos \frac{\pi s}{2}\sum_{m=0}^{\infty}\left(\frac{2m+1}{ny}\right)^{s-1}\sin \left( (2m+1)\pi \frac{x+ry}{ny}\right)\\
			&=\frac{2\Gamma(1-s)}{\pi^{1-s}}\frac{1}{y}\sin \frac{\pi s}{2}\sum_{m=0}^{\infty}\left(\frac{2m+1}{y}\right)^{s-1}\cos\left( (2m+1)\pi \frac{x}{y}\right) \\
			&\quad+\frac{2\Gamma(1-s)}{\pi^{1-s}}\frac{1}{y}\cos \frac{\pi s}{2}\sum_{m=0}^{\infty}\left(\frac{2m+1}{y}\right)^{s-1}\sin\left( (2m+1)\pi \frac{x}{y}\right) \\
			&\quad(\text{by Proposition}~ \ref{Pro 3.7}~ \text{and Eq.}~(\ref{star}))\\        
			&= y^{-s}\frac{2\Gamma(1-s)}{\pi^{1-s}}\left[\sin \frac{\pi s}{2}\sum_{m=0}^{\infty}\frac{\cos\left( (2m+1)\pi \frac{x}{y}\right)}{(2m+1)^{1-s}}+\cos \frac{\pi s}{2}\sum_{m=0}^{\infty}\frac{\sin\left( (2m+1)\pi \frac{x}{y}\right)}{(2m+1)^{1-s}}\right]\\
			&= y^{-s}\zeta_E\left( s,\frac{x}{y}\right).
		\end{align*}
		For $x\in \mathbb{C}\setminus\mathbb{Z}_{\leq0}$, from the difference equation of $\zeta_E(s,x):$
		$$\zeta_{E}(s,x+1)+\zeta_{E}(s,x)=x^{-s},$$
		we have
		\begin{align*}
		       y^{-s}\zeta_E\left( s,\frac{x+y}{y}\right) + y^{-s}\zeta_E\left( s,\frac{x}{y}\right)
			&= y^{-s}\zeta_E\left( s,1+\frac{x}{y}\right) + y^{-s}\zeta_E\left( s,\frac{x}{y}\right) \\
			 &= x^{-s}
		\end{align*}
		and
		\begin{align*}
			\lim_{a \to 0^+}a^{-s}\zeta_E\left(s,\frac{x+a}{a}\right) + a^{-s}\zeta_E\left( s,\frac{x}{a}\right) = x^{-s}.
		\end{align*}
		Therefore
		\begin{align*}
			y^{-s}\zeta_E\left( s,\frac{x+y}{y}\right) + y^{-s}\zeta_E\left( s,\frac{x}{y}\right) = \lim_{a \to 0^+}a^{-s}\zeta_E\left( s,\frac{x+a}{a}\right) + a^{-s}\zeta_E\left( s,\frac{x}{a}\right).
		\end{align*}
		Hence by Proposition \ref{Pro 3.8}, when $\textup{Re}(s)<0$ and $\frac{x}{y}\in\mathbb{C}\setminus\mathbb{Z}_{\leq0}$, we have $f(x,y)\in \mathrm{InV}^*$.
		
		Recall that for $\textup{Re}(s)>0$ and $x\in \mathbb{C}\setminus\mathbb{Z}_{\leq0}$,  
		\begin{align*}
			\zeta_E(s,x)=\sum_{m=0}^{\infty}\frac{(-1)^m}{(m+x)^s}.
		\end{align*}
		Thus
		\begin{align*}
			&\sum_{r=0}^{n-1}(-1)^r(ny)^{-s}\zeta_E\left( s,\frac{x+ry}{ny}\right) = \sum_{r=0}^{n-1}(-1)^r\sum_{m=0}^{\infty}(-1)^m\frac{1}{(ny)^s\left(m+\frac{x+ry}{ny}\right)^s}.
		\end{align*}
		Because $n$ is odd, we have
		\begin{align*}
			&\sum_{r=0}^{n-1}(-1)^r(ny)^{-s}\zeta_E\left( s,\frac{x+ry}{ny}\right)\\
			 = &\sum_{r=0}^{n-1}\sum_{m=0}^{\infty}(-1)^{nm+r}\left(\frac{1}{(nm+r)y+x}\right)^s \\    
			= &\sum_{m=0}^{\infty}(-1)^m\left(\frac{1}{my+x}\right)^s \\
			= &y^{-s}\sum_{m=0}^{\infty}(-1)^m\left(\frac{1}{m+\frac{x}{y}}\right)^s \\
			= &y^{-s}\zeta_E\left( s,\frac{x}{y}\right).
		\end{align*}
		This completes the proof.
	\end{proof}	
	In analogy with  Theorem \ref{Theorem 1.3}, we have the following convolution formulas for the alternating Hurwitz zeta function.
	\begin{theorem}\label{Theorem 1.6}
		For real numbers $\alpha, \beta > 1$, we have
		\begin{align*}
			\frac{y^{\alpha-1} \zeta_E\left(1 - \alpha, \frac{x}{y}\right)}{\Gamma(\alpha)} \ast \frac{y^{\beta-1} \zeta_E\left(1 - \beta, \frac{x}{y}\right)}{\Gamma(\beta)} = \frac{y^{\alpha + \beta - 1} \zeta_E\left(1 - \alpha - \beta, \frac{x}{y}\right)}{\Gamma(\alpha + \beta)}.
		\end{align*}
	\end{theorem}
	The following parts of this paper will be organized as follows. 
	In Section 2, we shall systematically investigate the fundamental properties of alternating invariant functions, proving that the set of alternating invariant functions is closed under the operation translation, reflection, and differentiation. 
	In Section 3, we shall derive convolution formulas for alternating invariant functions and prove their closure under convolution operations (Theorem \ref{Theorem 1.4}). In addition, we will prove convolution formulas for the Euler polynomials and the alternating Hurwitz zeta function, respectively (Theorems \ref{Theorem 1.5} and \ref{Theorem 1.6}).
	In Section 4, we shall provide further examples of alternating invariant functions, including the Gamma function $\tilde{\Gamma}(x)$ corresponding to the alternating Hurwitz zeta function, several appropriate combinations of trigonometric, exponential functions and  logarithmic functions, et al.	
	\section{Properties}
	In this section, we systematically investigate the properties of alternating invariant functions. 
In particular, Proposition~\ref{Pro 3.2} establishes the closure of alternating invariant functions 
under differentiation operations. Propositions~\ref{Pro 3.1}--\ref{Pro 3.4} can be obtained directly 
from Definition~\ref{stardef}.
		\begin{proposition}\label{Pro 3.1}
		If $f(x,y)\in \mathrm{InV}^*$, $a,b\in \mathbb{R}$ and $c>0$, then $F(x,y)=af(b+cx,cy)\in \mathrm{InV}^*$.  
	\end{proposition}

	\begin{proposition}\label{Pro 3.2}
		If $f(x,y)\in \mathrm{InV}^*$ and $\frac{\partial f}{\partial x}$ exists, then $\frac{\partial f}{\partial x}\in \mathrm{InV}^*$.  
	\end{proposition}
	 	
	\begin{proposition}\label{Pro 3.3}
		Let $F(x,y)=f(x,y)+ig(x,y)$, where $f(x,y)$ is the real part and $g(x,y)$ is the imaginary part. If $F(x,y)\in \mathrm{InV}^*$, then $f(x,y)$ and $g(x,y)\in \mathrm{InV}^*$.  
	\end{proposition}
	 	
	\begin{proposition}\label{Pro 3.4}
		If $f(x,y)\in \mathrm{InV}^*$, $m,n\in \mathbb{N}$ and $m,n$ are odd, then we have
		\begin{align*}
			\sum_{r=0}^{n-1}(-1)^r f(x+rmy,ny) = \sum_{r=0}^{m-1}(-1)^r f(x+rny,my).
		\end{align*}
	\end{proposition}
	 	
	\begin{proposition}\label{Pro 3.5}
		If $f(x,y)\in \mathrm{InV}^*$, then $F(x,y)=f(y-x,y)\in \mathrm{InV}^*$.  
	\end{proposition}
	\begin{proof}
		For $F(x,y)=f(y-x,y)$, we have
		\begin{align*}
			\sum_{r=0}^{n-1}(-1)^r F(x+ry,ny) 
			&= \sum_{r=0}^{n-1}(-1)^r f(ny-x-ry,ny) \\
			&= \sum_{r=0}^{n-1}(-1)^r f(y-x+ny-ry-y,ny) \\
			&= \sum_{r=0}^{n-1}(-1)^r f(y-x+(n-r-1)y,ny).
		\end{align*}
		Then since $n$ is odd, we see that
		\begin{align*}    
			\sum_{r=0}^{n-1}(-1)^r F(x+ry,ny)
			&= \sum_{r=0}^{n-1}(-1)^{n-1-r} f(y-x+(n-r-1)y,ny) \\
			&= \sum_{k=0}^{n-1}(-1)^k f(y-x+ky,ny) \\
			&= f(y-x,y),
		\end{align*}
		which shows that $F(x,y)=f(y-x,y)\in \mathrm{InV}^*$.
	\end{proof}
	
	Let $[x]$ be the greatest integer not exceeding $x.$ For $x\in\mathbb R$ let $\left\{ x\right\}$ be
	the fractional part of $x.$ That is, $\left\{ x\right\}=x-[x].$
	
	\begin{proposition}\label{Pro 3.7}
		If $f(x,y) \in \mathrm{InV}^*$ and $t\in\mathbb R,$ define
		\begin{equation}\label{(2-6-1)}
			\begin{aligned}
				f_1(x,y) &= (-1)^{\left[\frac{t+x}{y}\right]} f\left(y\left\{\frac{t+x}{y}\right\}, y\right),\\
				f_2(x,y) &= (-1)^{\left[\frac{t-x}{y}\right]} f\left(y\left\{\frac{t-x}{y}\right\}, y\right),
			\end{aligned}
		\end{equation}
		then we have $f_1(x,y)$ and  $f_2(x,y) \in \mathrm{InV}^*$.
	\end{proposition}
	
	\begin{proof}
		 Assuming $\left[\frac{t+x}{y}\right] = kn + m$  for some $k\in\mathbb{N}_{0}$ and $m\in\{0,1,\ldots,n-1\}$,
		by Eq. (\ref{(2-6-1)}) we have
		\begin{align*}
			&\sum_{r=0}^{n-1} (-1)^r f_1(x+ry, ny) \\
			=\,& (-1)^{\left[\frac{t+x}{y}\right]} \sum_{r=0}^{n-1-m} (-1)^{\left[\frac{t+x}{y}\right]} (-1)^r (-1)^{\left[\frac{\frac{t+x}{y}+r}{n}\right]} f\left(ny\left\{\frac{\left\{\frac{t+x}{y}\right\} + \left[\frac{t+x}{y}\right] + r}{n}\right\}, ny\right) \\
			&+ (-1)^{\left[\frac{t+x}{y}\right]} \sum_{r=n-m}^{n-1} (-1)^{\left[\frac{t+x}{y}\right]} (-1)^r (-1)^{\left[\frac{\frac{t+x}{y}+r}{n}\right]} f\left(ny\left\{\frac{\left\{\frac{t+x}{y}\right\} + \left[\frac{t+x}{y}\right] + r}{n}\right\}, ny\right).
		\end{align*}
		Since $\left[\frac{t+x}{y}\right] = kn + m$ and $n$ is odd, we distinguish  the following two cases:\\
		if $0 \leq r \leq n-1-m$, then
		$$
		0 \leq \left\{\frac{t+x}{y}\right\} + r + m \leq \left\{\frac{t+x}{y}\right\} + n-1-m+m < n,
		$$
		so we get
		$$
		\left(-1\right)^{\left[\frac{\frac{t+x}{y}+r}{n}\right]} = \left(-1\right)^{\left[\frac{\left\{\frac{t+x}{y}\right\} + kn + m + r}{n}\right]} = \left(-1\right)^{\left[\frac{\left\{\frac{t+x}{y}\right\} + m + r}{n} + k\right]} = (-1)^k
		$$
		and
		$$
		\left\{\frac{\left\{\frac{t+x}{y}\right\} + \left[\frac{t+x}{y}\right] + r}{n}\right\} = \left\{\frac{\left\{\frac{t+x}{y}\right\} + kn + m + r}{n}\right\} = \frac{\left\{\frac{t+x}{y}\right\} + m + r}{n};
		$$
		if $n-m \leq r \leq n-1$, then
		$$
		n \leq \left\{\frac{t+x}{y}\right\} + n-m+m \leq \left\{\frac{t+x}{y}\right\} + r + m < 2n,
		$$
		so we get
		$$
		(-1)^{\left[\frac{\frac{t+x}{y}+r}{n}\right]} = \left(-1\right)^{\left[\frac{\left\{\frac{t+x}{y}\right\} + kn + m + r}{n}\right]} = \left(-1\right)^{\left[\frac{\left\{\frac{t+x}{y}\right\} + m + r}{n} + k\right]} = (-1)^{k+1}
		$$
		and
		$$
		\left\{\frac{\left\{\frac{t+x}{y}\right\} + \left[\frac{t+x}{y}\right] + r}{n}\right\} = \left\{\frac{\left\{\frac{t+x}{y}\right\} + kn + m + r}{n}\right\} = \frac{\left\{\frac{t+x}{y}\right\} + m + r - n}{n}.
		$$
		Therefore,
		\begin{align*}
			&\sum_{r=0}^{n-1} (-1)^r f_1(x+ry, ny) \\
			=\,& (-1)^{\left[\frac{t+x}{y}\right]} \sum_{r=0}^{n-1-m} (-1)^{kn} (-1)^{r+m} (-1)^k f\left(ny \cdot \frac{\left\{\frac{t+x}{y}\right\} + m + r}{n}, ny\right) \\
			&+ (-1)^{\left[\frac{t+x}{y}\right]} \sum_{r=n-m}^{n-1} (-1)^{kn} (-1)^{r+m} (-1)^{k+1} f\left(ny \cdot \frac{\left\{\frac{t+x}{y}\right\} + m + r - n}{n}, ny\right).
		\end{align*}
		Since $n$ is odd, we have
		$$
		(-1)^{kn} (-1)^{r+m} (-1)^k = (-1)^{k(n+1)} (-1)^{r+m} = (-1)^{r+m}
		$$
		and
		$$
		(-1)^{kn} (-1)^{r+m} (-1)^{k+1} = (-1)^{r+m+1} = (-1)^{r+m-n}.
		$$
		Thus
		\begin{align*}
			&\sum_{r=0}^{n-1} (-1)^r f_1(x+ry, ny) \\
			=\,& (-1)^{\left[\frac{t+x}{y}\right]} \sum_{r=0}^{n-1-m} (-1)^{r+m} f\left(ny \cdot \frac{\left\{\frac{t+x}{y}\right\} + m + r}{n}, ny\right) \\
			&+ (-1)^{\left[\frac{t+x}{y}\right]} \sum_{r=n-m}^{n-1} (-1)^{r+m-n} f\left(ny \cdot \frac{\left\{\frac{t+x}{y}\right\} + m + r - n}{n}, ny\right)\\
			=\,& (-1)^{\left[\frac{t+x}{y}\right]} \sum_{s=m}^{n-1} (-1)^s f\left(ny \cdot \frac{\left\{\frac{t+x}{y}\right\} + s}{n}, ny\right) \\
			&+ (-1)^{\left[\frac{t+x}{y}\right]} \sum_{s=0}^{m-1} (-1)^s f\left(ny \cdot \frac{\left\{\frac{t+x}{y}\right\} + s}{n}, ny\right) \\
			=\,& (-1)^{\left[\frac{t+x}{y}\right]} \sum_{s=0}^{n-1} (-1)^s f\left(ny \cdot \frac{\left\{\frac{t+x}{y}\right\} + s}{n}, ny\right) \\
			=\,& (-1)^{\left[\frac{t+x}{y}\right]} \sum_{s=0}^{n-1} (-1)^s f\left(y \cdot \left\{\frac{t+x}{y}\right\} + sy, ny\right).
		\end{align*}
		Since $f(x,y) \in \mathrm{InV}^*$, we have
		$$
		\sum_{r=0}^{n-1} (-1)^r f_1(x+ry, ny) = (-1)^{\left[\frac{t+x}{y}\right]} f\left(y \cdot \left\{\frac{t+x}{y}\right\}, y\right).
		$$
		Thus $f_1(x,y) \in \mathrm{InV}^*$.
		Finally, since $f_2(x,y) = f_1(y-x, y)$, by Proposition \ref{Pro 3.5}, we have $f_2(x,y) \in \mathrm{InV}^*$.
	\end{proof}
	\begin{proposition}\label{Pro 3.7}
		Let $h(x)$ be a real function, and define
		\begin{equation}\label{(2-7-1)}
			\begin{aligned}
				f(x,y) &= \frac{1}{y}\sum_{m=0}^{\infty}h\left(\frac{2m+1}{y}\right)\cos\pi\frac{(2m+1)x}{y}, \\
				g(x,y) &= \frac{1}{y}\sum_{m=0}^{\infty}h\left(\frac{2m+1}{y}\right)\sin\pi\frac{(2m+1)x}{y}.
			\end{aligned}
		\end{equation}
		Then we have $f(x,y)$ and $g(x,y)\in \mathrm{InV}^*$.  
	\end{proposition}
	\begin{proof}
		Let $$F(x,y)=\frac{1}{y}\sum_{m=0}^{\infty}h\left(\frac{2m+1}{y}\right)e^{\pi i\frac{(2m+1)x}{y}}.$$
		For any odd positive integer $n$, we have
		\begin{align*}
			\sum_{r=0}^{n-1}(-1)^r F(x+ry,ny) 
			&= \frac{1}{ny}\sum_{r=0}^{n-1}(-1)^r \sum_{m=0}^{\infty}h\left(\frac{2m+1}{ny}\right)e^{\pi i\frac{(2m+1)(x+ry)}{ny}} \\    
			&= \frac{1}{ny}\sum_{m=0}^{\infty}h\left(\frac{2m+1}{ny}\right)e^{\pi i\frac{(2m+1)x}{ny}}\sum_{r=0}^{n-1}(-1)^r e^{\pi i\frac{(2m+1)r}{n}}.
		\end{align*}
		Since
		\begin{align*}
			\sum_{r=0}^{n-1}(-1)^r e^{\pi i\frac{(2m+1)r}{n}} = 
			\begin{cases}
				0, & \textrm{if}~n\nmid 2m+1, \\
				n, & \textrm{if}~n\mid 2m+1,
			\end{cases}
		\end{align*}
		we have
		\begin{align*}
			\sum_{r=0}^{n-1}(-1)^r F(x+ry,ny) 
			&= \frac{1}{ny}\sum_{m=0}^{\infty}h\left(\frac{2m+1}{y}\right)e^{\pi i\frac{(2m+1)x}{y}}\cdot n \\
			&= \frac{1}{y}\sum_{m=0}^{\infty}h\left(\frac{2m+1}{y}\right)e^{\pi i\frac{(2m+1)x}{y}} \\
			& = F(x,y).
		\end{align*}
		Then by Euler's identity $e^{i\theta}=\cos\theta+i\sin\theta$, we have $F(x,y)=f(x,y)+ig(x,y),$ where $f(x,y)$ and $g(x,y)$
		are given in (\ref{(2-7-1)}). Finally, from Proposition \ref{Pro 3.3}, we conclude that $f(x,y)$ and $g(x,y)\in \mathrm{InV}^*$. 
	\end{proof}

	\begin{proposition}\label{Pro 3.8}
		For a real function $f(x,y)$ with $y>0$, $f(x,y)\in \mathrm{InV}^*$ if and only if the equality
		\begin{equation}\label{(3-1)}
			\begin{aligned}
				\sum_{r=0}^{n-1}(-1)^r f(x+ry,ny) = f(x,y)
			\end{aligned}
		\end{equation}
		holds for all $0<x\leq y$ and for all $x\in \mathbb{R}$, we have
		\begin{equation}\label{(3-1-2)}
			\begin{aligned}
				f(x+y,y)+f(x,y)&=\lim_{n \to \infty}f\left(x+\frac{y}{n},\frac{y}{n}\right)+f\left(x,\frac{y}{n}\right)
				\\&=\lim_{a \to 0^+}f(x+a,a)+f(x,a).
			\end{aligned}
		\end{equation}
	\end{proposition}
	\begin{proof}
		The necessity is obvious.
		
		For sufficiency: by (\ref{(3-1-2)}), for any $0<x\leq y$ and odd $n$, we have
		\begin{align*}
			f(x+ny,ny)+f(x,ny) 
			= \lim_{n \to \infty}f\left(x+\frac{ny}{n},\frac{ny}{n}\right) + \lim_{n \to \infty}f\left(x,\frac{ny}{n}\right),
		\end{align*}
		which is equivalent to
		\begin{equation}\label{(3-2)}
			\begin{aligned}
				f(x+ny,ny)+f(x,ny)=f(x+y,y)+f(x,y).
			\end{aligned}
		\end{equation}
		Thus by (\ref{(3-1)}),
		\begin{align*}
			&\sum_{r=0}^{n-1}(-1)^r f(x+y+ry,ny) \\
			=\,& f(x+ny,ny) + f(x,ny) - \sum_{r=0}^{n-1}(-1)^r f(x+ry,ny) \\
			=\,& f(x+y,y) + f(x,y) - f(x,y) \\
			=\,& f(x+y,y),
		\end{align*}
		which shows that (\ref{(3-1)}) holds for $x>y$.
		
		From (\ref{(3-2)}) we also have
		\begin{align*}
			&\sum_{r=0}^{n-1}(-1)^r f(x-y+ry,ny) \\
			=\,& f(x-y+ny,ny) + f(x-y,ny) - \sum_{r=0}^{n-1}(-1)^r f(x+ry,ny) \\
			=\,& f(x-y+y,y) + f(x-y,y) - f(x,y) \\
			=\,& f(x-y,y),
		\end{align*}
		showing that (\ref{(3-1)}) holds for $x\leq 0$.
		
		In conclusion,  (\ref{(3-1)}) holds for all $x\in\mathbb{R}$, which means that $f(x,y)\in \mathrm{InV}^*$.
	\end{proof}
	
	\section{Convolution formulas}
	In this section, we will prove the main theorems of this paper: Theorems \ref{Theorem 1.4}, \ref{Theorem 1.5}, and \ref{Theorem 1.6}.
	
	\subsection*{Proof of Theorem \ref{Theorem 1.4}}
	Define the functions
	\begin{align*}
		&f(x,y) = g \ast h(x,y), \\
		&h_1(x,y) = h(y - x, y)
	\end{align*}
	and  \begin{equation}\label{(1-7-1)}
		f_1(x,y) = \int_{0}^{y} (-1)^{\left[\frac{t+x}{y}\right]} h_1\left(y\left\{\frac{t+x}{y}\right\}, y\right) g(t,y) \, dt.
	\end{equation}
	Then we have
	\begin{align*}
		&\sum_{r=0}^{n-1} (-1)^r f_1(x + ry, ny) \\
		=\,& \int_{0}^{ny} \sum_{r=0}^{n-1} (-1)^r (-1)^{\left[\frac{t + x + ry}{ny}\right]} h_1\left(ny\left\{\frac{t + x + ry}{ny}\right\}, y\right) g(t, ny) \, dt.
	\end{align*}
	By Proposition \ref{Pro 3.5}, we know that $h_1(x,y) \in \mathrm{InV}^*$. Thus, from the above equation, we get
	\begin{align*}
		&\sum_{r=0}^{n-1} (-1)^r f_1(x + ry, ny) \\
		=\,& \int_{0}^{ny} (-1)^{\left[\frac{t + x}{y}\right]} h_1\left(y\left\{\frac{t + x}{y}\right\}, y\right) g(t, ny) \, dt \\
		=\,& \sum_{r=0}^{n-1} \int_{ry}^{(r+1)y} (-1)^{\left[\frac{t + x}{y}\right]} h_1\left(y\left\{\frac{t + x}{y}\right\}, y\right) g(t, ny) \, dt.
	\end{align*}
	Then by substituting $u = t - ry$, we obtain
	\begin{align*}
		&\sum_{r=0}^{n-1} (-1)^r f_1(x + ry, ny) \\
		=\,& \sum_{r=0}^{n-1} \int_{0}^{y} (-1)^{\left[\frac{x + u + ry}{y}\right]} h_1\left(y\left\{\frac{x + u + ry}{y}\right\}, y\right) g(u + ry, ny) \, du \\
		=\,& \sum_{r=0}^{n-1} \int_{0}^{y} (-1)^r (-1)^{\left[\frac{x + u}{y}\right]} h_1\left(y\left\{\frac{x + u}{y}\right\}, y\right) g(u + ry, ny) \, du.
	\end{align*}
	Since $g(x,y) \in \mathrm{InV}^*$, the above expression simplifies to
	\begin{align*}
		&\sum_{r=0}^{n-1} (-1)^r f_1(x + ry, ny) \\
		=\,& \int_{0}^{y} (-1)^{\left[\frac{x + u}{y}\right]} h_1\left(y\left\{\frac{x + u}{y}\right\}, y\right) g(u, y) \, du \\
		=\,& f_1(x, y),
	\end{align*}
	which implies $f_1(x,y) \in \mathrm{InV}^*$. \\
	
	For $0 < x \leq y$, by (\ref{(1-7-1)}) we have
	\begin{align*}
		f_1(x,y) =& \int_{0}^{y} (-1)^{\left[\frac{t + x}{y}\right]} h_1\left(y\left\{\frac{t + x}{y}\right\}, y\right) g(t,y) \, dt \\
		=& \int_{0}^{y - x} (-1)^{\left[\frac{t + x}{y}\right]} h_1\left(y\left\{\frac{t + x}{y}\right\}, y\right) g(t,y) \, dt \\
		&+ \int_{y - x}^{y} (-1)^{\left[\frac{t + x}{y}\right]} h_1\left(y\left\{\frac{t + x}{y}\right\}, y\right) g(t,y) \, dt \\
		=& \int_{0}^{y - x} h_1(t + x, y) g(t,y) \, dt - \int_{y - x}^{y} h_1(t + x - y, y) g(t,y) \, dt \\
		=& \int_{0}^{y - x} h(y - x - t, y) g(t,y) \, dt - \int_{y - x}^{y} h(2y - x - t, y) g(t,y) \, dt \\
		=& f(y - x, y),
	\end{align*}
	i.e., $f_1(x,y) = f(y - x, y)$. Then, by Proposition \ref{Pro 3.5}, for $0 < x \leq y$ and $n$ odd, we have the equality
	\begin{equation}\label{(4-1)}
		\begin{aligned}
			\sum_{r=0}^{n-1} (-1)^r f(x + ry, ny) = f(x, y).
		\end{aligned}
	\end{equation}
	Next, we prove that for any $x \in \mathbb{R}$ and $y > 0$, $f(x,y)$ also satisfies the above equation. \\
	
	Let $\bar{f}(x,y) = f(y + x, y) + f(x, y)$. By the convolution formula (\ref{(1-10)}), we have
	\begin{align*}
		\bar{f}(x,y) =& \int_{0}^{x + y} g(t,y) h(x + y - t, y) \, dt - \int_{x + y}^{y} g(t,y) h(x + 2y - t, y) \, dt \\
		&+ \int_{0}^{x} g(t,y) h(x - t, y) \, dt - \int_{x}^{y} g(t,y) h(x + y - t, y) \, dt \\
		=& \int_{y}^{x + y} g(t,y) h(x + y - t, y) \, dt + \int_{y}^{x + y} g(t,y) h(x + 2y - t, y) \, dt \\
		&+ \int_{0}^{y} g(t,y) h(x + y - t, y) \, dt + \int_{y}^{x} g(t,y) h(x + y - t, y) \, dt \\
		&+ \int_{0}^{x} g(t,y) h(x - t, y) \, dt \\
		=& \int_{y}^{x + y} g(t,y) \left[h(x + y - t, y) + h(x + 2y - t, y)\right] \, dt \\
		&+ \int_{0}^{x} g(t,y) \left[h(x + y - t, y) + h(x - t, y)\right] \, dt,
	\end{align*}
	then by substituting $t \to t - y$, we obtain
	\begin{equation}\label{(4-2)}
		\begin{aligned}
			\bar{f}(x,y) =& \int_{0}^{x} g(t + y, y) \left[h(x - t, y) + h(x + y - t, y)\right] \, dt \\
			&+ \int_{0}^{x} g(t, y) \left[h(x + y - t, y) + h(x - t, y)\right] \, dt \\
			=& \int_{0}^{x} \bar{g}(t,y) \bar{h}(x - t, y) \, dt.
		\end{aligned}
	\end{equation}
	Let $\bar{F}(x,y) = F(y + x, y) + F(x, y)$. Then, for any $F(x,y) \in \mathrm{InV}^*$, we have
	\begin{align*}
		\bar{F}(x, ny) &= F(ny + x, ny) + F(x, ny) \\
		&= \sum_{r=0}^{n-1} (-1)^r F(x + y + ry, ny) + \sum_{r=0}^{n-1} (-1)^r F(x + ry, ny) \\
		&= F(x + y, y) + F(x, y) \\
		&= \bar{F}(x, y).
	\end{align*}
	Since $g(x,y), h(x,y) \in \mathrm{InV}^*$, from the above and Eq. (\ref{(4-2)}), we have
	\begin{equation}\label{(4-3)}
		\begin{aligned}
			\bar{f}(x, ny) = \int_{0}^{x} \bar{g}(t, ny) \bar{h}(x - t, ny) \, dt = \int_{0}^{x} \bar{g}(t, y) \bar{h}(x - t, y) \, dt = \bar{f}(x, y).
		\end{aligned}
	\end{equation}
	Note that for $0 < x \leq y$, Eq. (\ref{(4-1)}) holds. Thus, using Eqs. (\ref{(4-1)}) and (\ref{(4-3)}), we obtain
	\begin{align*}
		&\sum_{r=0}^{n-1} (-1)^r f(x + y + ry, ny) \\
		=\,& -\sum_{r=0}^{n-1} (-1)^r f(x + ry, ny) + f(x, ny) + f(x + ny, ny) \\
		=\,& -f(x, y) + \bar{f}(x, ny) \\
		=\,& -f(x, y) + \bar{f}(x, y) \\
		=\,& -f(x, y) + f(x, y) + f(x + y, y) \\
		=\,& f(x + y, y),
	\end{align*}
	i.e., for $x > y$, Eq. (\ref{(4-1)}) also holds. \\
	
	Finally, by Eqs. (\ref{(4-1)}) and (\ref{(4-3)}), we have
	\begin{align*}
		&\sum_{r=0}^{n-1} (-1)^r f(x - y + ry, ny) \\
		=\,& -\sum_{r=0}^{n-1} (-1)^r f(x + ry, ny) + f(x + (n-1)y, ny) + f(x - y, ny) \\
		=\,& -f(x, y) + \bar{f}(x - y, ny) \\
		=\,& -f(x, y) + \bar{f}(x - y, y) \\
		=\,& -f(x, y) + f(x, y) + f(x - y, y) \\
		=\,& f(x - y, y),
	\end{align*}
	i.e., for $x \leq 0$, Eq. (\ref{(4-1)}) also holds. 
	In conclusion, for any $x \in \mathbb{R}$, Eq. (\ref{(4-1)}) holds, i.e., $f(x,y) \in \mathrm{InV}^*$.

	\subsection*{Proof of Theorem \ref{Theorem 1.5}}
	Next, we present the proof of Theorem \ref{Theorem 1.5}, which requires the following lemmas.
	\begin{lemma}\label{Pro 2.12}\textup{(Reflection equation for Euler polynomials, see \cite[Theorem 1.2]{2})} 
		Let $n \in \mathbb{N},$ we have
		\begin{align*}
			\sum_{k=0}^{n} \binom{n}{k} E_k(x) + E_n(x) = 2 x^n.
		\end{align*}
		In particular, we have
		\begin{align*}
			E_n(1-x) = (-1)^n E_n(x).
		\end{align*}
	\end{lemma}
	
	\begin{lemma}[{See \cite[p. 31]{MOS}}]\label{Lemma 4.1}
		For $m, n \in \mathbb{N}$, we have
		\begin{align*}
			&\int_{0}^{1} E_m(t) E_n(1 - t) \, dt = (-1)^n \int_{0}^{1} E_m(t) E_n(t) \, dt = \frac{2 E_{m+n+1}(1)}{\binom{m+n}{n} (m + n + 1)}.
		\end{align*}
	\end{lemma}
	 	
	Now we are at the position to prove Theorem  \ref{Theorem 1.5}.
	
	\begin{proof}[Proof of Theorem \ref{Theorem 1.5}]
		We first prove that the Euler polynomials satisfy the following identity:
		\begin{equation}\label{(1-5-1)}
			\begin{aligned}
				E_{m + n + 1}(x) = \frac{\binom{m + n}{m} (m + n + 1)}{2} \left( \int_{0}^{1} E_m(x - t) E_n(t) \, dt - 2 \int_{x}^{1} (x - t)^m E_n(t) \, dt \right).
			\end{aligned}
		\end{equation}
		To this end, we use the induction method for $m\in\mathbb{N}$. For $m = 1$, differentiating the right-hand side of the above equation gives
		\begin{align*}
			&\frac{d}{dx} \left( \int_{0}^{1} E_1(x - t) E_n(t) \, dt - 2 \int_{x}^{1} (x - t) E_n(t) \, dt \right) \\
			=\,& \int_{0}^{1} E_n(t) \, dt - 2 \int_{x}^{1} E_n(t) \, dt \\
			=\,& \frac{2 E_{n + 1}(1)}{n + 1} - 2 \left( \frac{E_{n + 1}(1) - E_{n + 1}(x)}{n + 1} \right) \\
			=\,& \frac{2 E_{n + 1}(x)}{n + 1} \\
			=\,& \frac{d}{dx} \left( \frac{2 E_{n + 2}(x)}{(n + 1)(n + 2)} \right).
		\end{align*}
		And by Lemma \ref{Lemma 4.1}, at $x=1$ we have
		\begin{align*}
			\int_{0}^{1} E_1(1 - t) E_n(t) \, dt = \frac{2 E_{n + 2}(1)}{\binom{1 + n}{n} (n + 2)} = \frac{2 E_{n + 2}(1)}{(n + 1)(n + 2)},
		\end{align*}
		and thus
		\begin{align*}
			\int_{0}^{1} E_1(x - t) E_n(t) \, dt - 2 \int_{x}^{1} (x - t) E_n(t) \, dt = \frac{2 E_{n + 2}(x)}{(n + 1)(n + 2)}.
		\end{align*}
		This shows that the equation holds for $m = 1$. 
		Assume that the equation holds for $m = k$. We now prove the case for $m = k + 1$. Differentiating the right-hand side of Eq. (\ref{(1-5-1)}) gives
		\begin{align*}
			&\frac{d}{dx} \left( \int_{0}^{1} E_{k + 1}(x - t) E_n(t) \, dt - 2 \int_{x}^{1} (x - t)^{k + 1} E_n(t) \, dt \right) \\
			=\,& (k + 1) \left( \int_{0}^{1} E_k(x - t) E_n(t) \, dt - 2 \int_{x}^{1} (x - t)^k E_n(t) \, dt \right) \\
			=\,& (k + 1) \frac{2 E_{k + n + 1}(x)}{\binom{k + n}{k} (k + n + 1)} \\
			=\,& \frac{2 E_{k + n + 1}(x)}{\binom{k + n + 1}{k + 1}} \\
			=&\, \frac{d}{dx} \left( \frac{2 E_{k + n + 2}(x)}{\binom{k + n + 1}{k + 1} (k + n + 2)} \right).
		\end{align*}
		And by Lemma \ref{Lemma 4.1}, at $x=1$ we have
		\begin{align*}
			\int_{0}^{1} E_{k + 1}(1 - t) E_n(t) \, dt = \frac{2 E_{k + n + 2}(1)}{\binom{k + n + 1}{k + 1} (k + n + 2)},
		\end{align*}
		thus
		\begin{align*}
			\int_{0}^{1} E_{k + 1}(x - t) E_n(t) \, dt - 2 \int_{x}^{1} (x - t)^{k + 1} E_n(t) \, dt = \frac{2 E_{k + n + 2}(x)}{\binom{k + n + 1}{k + 1} (k + n + 2)}.
		\end{align*}
		By induction, Eq. (\ref{(1-5-1)}) holds for all positive integers $m$. 
		Next, we prove the convolution identity:
		\begin{align*}
			\frac{y^m E_m\left(\frac{x}{y}\right)}{2 m!} \ast \frac{y^n E_n\left(\frac{x}{y}\right)}{2 n!} = \frac{y^{m + n + 1} E_{m + n + 1}\left(\frac{x}{y}\right)}{2 (m + n + 1)!}.
		\end{align*}
		By the convolution formula (\ref{(1-10)}), we have
		\begin{align*}
			&\frac{y^m E_m\left(\frac{x}{y}\right)}{2 m!} \ast \frac{y^n E_n\left(\frac{x}{y}\right)}{2 n!} \\
			=\,& \int_{0}^{x} \frac{y^m E_m\left(\frac{t}{y}\right)}{2 m!} \cdot \frac{y^n E_n\left(\frac{x - t}{y}\right)}{2 n!} \, dt - \int_{x}^{y} \frac{y^m E_m\left(\frac{t}{y}\right)}{2 m!} \cdot \frac{y^n E_n\left(\frac{x + y - t}{y}\right)}{2 n!} \, dt \\
			=\,& \frac{y^{m + n}}{4 m! n!} \left( \int_{0}^{x} E_m\left(\frac{t}{y}\right) E_n\left(\frac{x - t}{y}\right) \, dt - \int_{x}^{y} E_m\left(\frac{t}{y}\right) E_n\left(\frac{x + y - t}{y}\right) \, dt \right) \\
			=\,& \frac{y^{m + n}}{4 m! n!} \left( \int_{0}^{y} E_m\left(\frac{t}{y}\right) E_n\left(\frac{x - t}{y}\right) \, dt - \int_{x}^{y} E_m\left(\frac{t}{y}\right) E_n\left(\frac{x - t}{y}\right) \, dt \right) \\
			&- \frac{y^{m + n}}{4 m! n!} \left( \int_{x}^{y} E_m\left(\frac{t}{y}\right) E_n\left(\frac{x + y - t}{y}\right) \, dt \right) \\
			=\,& \frac{y^{m + n}}{4 m! n!} \left( \int_{0}^{y} E_m\left(\frac{t}{y}\right) E_n\left(\frac{x - t}{y}\right) \, dt \right) \\
			&- \frac{y^{m + n}}{4 m! n!} \left( \int_{x}^{y} E_m\left(\frac{t}{y}\right) \left( E_n\left(\frac{x - t}{y}\right) + E_n\left(\frac{x + y - t}{y}\right) \right) \, dt \right).
		\end{align*}
		Then by the difference equation $E_n(x + 1) + E_n(x) = 2x^n$, we further have
		\begin{align*}
			&\frac{y^m E_m\left(\frac{x}{y}\right)}{2 m!} \ast \frac{y^n E_n\left(\frac{x}{y}\right)}{2 n!} \\
			=\,& \frac{y^{m + n}}{4 m! n!} \left( \int_{0}^{y} E_m\left(\frac{t}{y}\right) E_n\left(\frac{x - t}{y}\right) \, dt - 2 \int_{x}^{y} E_m\left(\frac{t}{y}\right) \left(\frac{x - t}{y}\right)^n \, dt \right).
		\end{align*}
		Finally, by substituting $u = \frac{t}{y}$, we obtain
		\begin{align*}
			&\frac{y^m E_m\left(\frac{x}{y}\right)}{2 m!} \ast \frac{y^n E_n\left(\frac{x}{y}\right)}{2 n!} \\
			=\,& \frac{y^{m + n}}{4 m! n!} \left( y \int_{0}^{1} E_m(u) E_n\left(\frac{x}{y} - u\right) \, du - 2 y \int_{\frac{x}{y}}^{1} E_m(u) \left(\frac{x}{y} - u\right)^n \, du \right) \\
			=\,& \frac{y^{m + n + 1}}{4 m! n!} \cdot E_{m + n + 1}\left(\frac{x}{y}\right) \cdot \frac{2}{\binom{m + n}{n} (m + n + 1)} \\
			=\,& \frac{y^{m + n + 1} E_{m + n + 1}\left(\frac{x}{y}\right)}{2 (m + n + 1)!}.
		\end{align*}
		This completes the proof.
	\end{proof}
	
	\subsection*{Proof of Theorem \ref{Theorem 1.6}}
		Let
		$f(x,y)=\frac{y^{\alpha-1}\zeta_E\left(1-\alpha,\frac{x}{y}\right)}{\Gamma(\alpha)}\ast \frac{y^{\beta-1}\zeta_E\left(1-\beta,\frac{x}{y}\right)}{\Gamma(\beta)}.$
		For real numbers $\alpha,\beta>1$, according to the convolution formula (\ref{(1-10)}), we have
		\begin{align*}
			f(x,y)=&\int_{0}^{x}\frac{y^{\alpha-1}\zeta_E\left(1-\alpha,\frac{t}{y}\right)}{\Gamma(\alpha)}\cdot\frac{y^{\beta-1}\zeta_E\left(1-\beta,\frac{x-t}{y}\right)}{\Gamma(\beta)}dt\\
			&-\int_{x}^{y}\frac{y^{\alpha-1}\zeta_E\left(1-\alpha,\frac{t}{y}\right)}{\Gamma(\alpha)}\cdot\frac{y^{\beta-1}\zeta_E\left(1-\beta,1+\frac{x-t}{y}\right)}{\Gamma(\beta)}dt.
		\end{align*}
		Through interval partitioning, we derive
		\begin{align*}
			f(x+y,y)+f(x,y)
			&=\int_{0}^{x+y}\frac{y^{\alpha-1}\zeta_E\left(1-\alpha,\frac{t}{y}\right)}{\Gamma(\alpha)}\cdot\frac{y^{\beta-1}\zeta_E\left(1-\beta,1+\frac{x-t}{y}\right)}{\Gamma(\beta)}dt\\
			&\quad-\int_{x+y}^{y}\frac{y^{\alpha-1}\zeta_E\left(1-\alpha,\frac{t}{y}\right)}{\Gamma(\alpha)}\cdot\frac{y^{\beta-1}\zeta_E\left(1-\beta,2+\frac{x-t}{y}\right)}{\Gamma(\beta)}dt\\
			&\quad+\int_{0}^{x}\frac{y^{\alpha-1}\zeta_E\left(1-\alpha,\frac{t}{y}\right)}{\Gamma(\alpha)}\cdot\frac{y^{\beta-1}\zeta_E\left(1-\beta,\frac{x-t}{y}\right)}{\Gamma(\beta)}dt\\
			&\quad-\int_{x}^{y}\frac{y^{\alpha-1}\zeta_E\left(1-\alpha,\frac{t}{y}\right)}{\Gamma(\alpha)}\cdot\frac{y^{\beta-1}\zeta_E\left(1-\beta,1+\frac{x-t}{y}\right)}{\Gamma(\beta)}dt\\
			&=\int_{0}^{x}\frac{y^{\alpha-1}\zeta_E\left(1-\alpha,\frac{t}{y}\right)}{\Gamma(\alpha)}\cdot\frac{y^{\beta-1}\zeta_E\left(1-\beta,1+\frac{x-t}{y}\right)}{\Gamma(\beta)}dt\\
			&\quad+\int_{0}^{x}\frac{y^{\alpha-1}\zeta_E\left(1-\alpha,\frac{t}{y}\right)}{\Gamma(\alpha)}\cdot\frac{y^{\beta-1}\zeta_E\left(1-\beta,\frac{x-t}{y}\right)}{\Gamma(\beta)}dt\\
			&\quad+\int_{y}^{x+y}\frac{y^{\alpha-1}\zeta_E\left(1-\alpha,\frac{t}{y}\right)}{\Gamma(\alpha)}\cdot\frac{y^{\beta-1}\zeta_E\left(1-\beta,1+\frac{x-t}{y}\right)}{\Gamma(\beta)}dt\\
			&\quad+\int_{y}^{x+y}\frac{y^{\alpha-1}\zeta_E\left(1-\alpha,\frac{t}{y}\right)}{\Gamma(\alpha)}\cdot\frac{y^{\beta-1}\zeta_E\left(1-\beta,2+\frac{x-t}{y}\right)}{\Gamma(\beta)}dt.
		\end{align*}
		Using the difference equation
		$$\zeta_E(s,1+x)+\zeta_E(s,x)=x^{-s},$$
		we can simplify the above expression to
		\begin{align*}
			f(x+y,y)+f(x,y)
			=&\int_{y}^{x+y}\frac{y^{\alpha-1}\zeta_E\left(1-\alpha,\frac{t}{y}\right)}{\Gamma(\alpha)}\cdot \frac{y^{\beta-1}}{\Gamma(\beta)\cdot\left(1+\frac{x-t}{y}\right)^{1-\beta}} dt\\
			&+\int_{0}^{x}\frac{y^{\alpha-1}\zeta_E\left(1-\alpha,\frac{t}{y}\right)}{\Gamma(\alpha)}\cdot \frac{1}{\Gamma(\beta)\left(x-t\right)^{1-\beta}} dt.
		\end{align*}
		Then  
		\begin{align*}
			f(x+y,y)+f(x,y)
			&= \int_{0}^{x}\frac{y^{\alpha-1}\zeta_E\left(1-\alpha,1+\frac{t}{y}\right)}{\Gamma(\alpha)} \cdot \frac{1}{\Gamma(\beta)(x-t)^{1-\beta}} dt \\
			&\quad (\text{letting}~ t\to t+y)\\
			&\quad + \int_{0}^{x}\frac{y^{\alpha-1}\zeta_E\left(1-\alpha,\frac{t}{y}\right)}{\Gamma(\alpha)} \cdot \frac{1}{\Gamma(\beta)(x-t)^{1-\beta}} dt \\
			&= \int_{0}^{x} \frac{y^{\alpha-1}}{\Gamma(\alpha)\Gamma(\beta)(x-t)^{1-\beta}} \left[\zeta_E\left(1-\alpha,1+\tfrac{t}{y}\right) + \zeta_E\left(1-\alpha,\tfrac{t}{y}\right)\right] dt \\
			&= \int_{0}^{x} \frac{y^{\alpha-1}}{\Gamma(\alpha)\Gamma(\beta)(x-t)^{1-\beta}} \cdot \left(\frac{t}{y}\right)^{\alpha-1} dt\\
			& \quad (\text{by}~\zeta_E(1-\alpha,a+1)+\zeta_E(1-\alpha,a) = a^{\alpha-1})\\
			&= \int_{0}^{x} \frac{1}{\Gamma(\alpha)\Gamma(\beta)} \cdot \frac{t^{\alpha-1}}{(x-t)^{1-\beta}} dt\\
			&\quad (\text{deleting the term}~y^{\alpha-1})\\
			&= \frac{x^{\alpha+\beta-1}}{\Gamma(\alpha)\Gamma(\beta)} \int_{0}^{1} s^{\alpha-1}(1-s)^{\beta-1} ds\\
			&\quad (\text{letting}~ t=xs)\\
			&= \frac{x^{\alpha+\beta-1}}{\Gamma(\alpha+\beta)}\\
			&\quad \left(\text{since}~B(\alpha,\beta) = \frac{\Gamma(\alpha)\Gamma(\beta)}{\Gamma(\alpha+\beta)}\right)\\
			&= \frac{y^{\alpha+\beta-1}}{\Gamma(\alpha+\beta)} \left[\zeta_E\left(1-\alpha-\beta,1+\tfrac{x}{y}\right) + \zeta_E\left(1-\alpha-\beta,\tfrac{x}{y}\right)\right]\\
			& \quad \left(\text{by}~\zeta_E\left(1-\alpha-\beta,1+\tfrac{x}{y}\right)+\zeta_E\left(1-\alpha-\beta,\tfrac{x}{y}\right) = \left(\tfrac{y}{x}\right)^{\alpha+\beta-1}\right).
		\end{align*}
		Thus for any $x\in\mathbb{R}$ we have
		\begin{equation}\label{(4-4)}
			\begin{aligned}
				f(x+y,y)+f(x,y)=\frac{y^{\alpha+\beta-1}}{\Gamma(\alpha+\beta)}\left(\zeta_E\left(1-\alpha-\beta,1+\frac{x}{y}\right)+\zeta_E\left(1-\alpha-\beta,\frac{x}{y}\right)\right).
			\end{aligned}
		\end{equation}
		From Euler's identity $e^{i\theta}=\cos\theta+i\sin\theta$, for $s>1$ and $0<x\leq y$, (\ref{zetaFour}) can be written as
		\begin{equation}\label{(1-6-1)}
			\begin{aligned}
				&y^{s-1}\zeta_E\left(1-s,\frac{x}{y}\right)\\
				=&\frac{\Gamma(s)}{\pi^s y}\left(e^{-\frac{is\pi}{2}}\sum_{m=0}^{\infty}\left(\frac{2m+1}{y}\right)^{-s}e^{\pi i(2m+1)\frac{x}{y}}+e^{\frac{is\pi}{2}}\sum_{m=0}^{\infty}\left(\frac{2m+1}{y}\right)^{-s}e^{-\pi i(2m+1)\frac{x}{y}}\right).
			\end{aligned}
		\end{equation}
		Thus for $0<x\leq y$, we obtain the Fourier expansion of $f(x,y)$:
		\begin{align*}
			f(x,y)
			&=\frac{1}{\pi^{\alpha+\beta}y^2}\\
			&\quad\times\int_{0}^{x}\left(e^{-\frac{i\alpha\pi}{2}}\sum_{m=0}^{\infty}\left(\frac{2m+1}{y}\right)^{-\alpha}e^{\pi i(2m+1)\frac{t}{y}}+e^{\frac{i\alpha\pi}{2}}\sum_{m=0}^{\infty}\left(\frac{2m+1}{y}\right)^{-\alpha}e^{-\pi i(2m+1)\frac{t}{y}}\right)\\
			&\quad\times\left(e^{-\frac{i\beta\pi}{2}}\sum_{n=0}^{\infty}\left(\frac{2n+1}{y}\right)^{-\beta}e^{\pi i(2n+1)\frac{x-t}{y}}+e^{\frac{i\beta\pi}{2}}\sum_{n=0}^{\infty}\left(\frac{2n+1}{y}\right)^{-\beta}e^{-\pi i(2n+1)\frac{x-t}{y}}\right)dt\\
			&\quad-\frac{1}{\pi^{\alpha+\beta}y^2}\\
			&\quad\times\int_{x}^{y}\left(e^{-\frac{i\alpha\pi}{2}}\sum_{m=0}^{\infty}\left(\frac{2m+1}{y}\right)^{-\alpha}e^{\pi i(2m+1)\frac{t}{y}}+e^{\frac{i\alpha\pi}{2}}\sum_{m=0}^{\infty}\left(\frac{2m+1}{y}\right)^{-\alpha}e^{-\pi i(2m+1)\frac{t}{y}}\right)\\
			&\quad\times\left(e^{-\frac{i\beta\pi}{2}}\sum_{n=0}^{\infty}\left(\frac{2n+1}{y}\right)^{-\beta}e^{\pi i(2n+1)\frac{x+y-t}{y}}+e^{\frac{i\beta\pi}{2}}\sum_{n=0}^{\infty}\left(\frac{2n+1}{y}\right)^{-\beta}e^{-\pi i(2n+1)\frac{x+y-t}{y}}\right)dt.
		\end{align*}
		Since $e^{\pi i(2m+1)}=-1$, we have
		\begin{align*}
			f(x,y)=&\frac{1}{\pi^{\alpha+\beta}y^2}\\
			&\times\int_{0}^{y}\left(e^{-\frac{i\alpha\pi}{2}}\sum_{m=0}^{\infty}\left(\frac{2m+1}{y}\right)^{-\alpha}e^{\pi i(2m+1)\frac{t}{y}}+e^{\frac{i\alpha\pi}{2}}\sum_{m=0}^{\infty}\left(\frac{2m+1}{y}\right)^{-\alpha}e^{-\pi i(2m+1)\frac{t}{y}}\right)\\
			&\times\left(e^{-\frac{i\beta\pi}{2}}\sum_{n=0}^{\infty}\left(\frac{2n+1}{y}\right)^{-\beta}e^{\pi i(2n+1)\frac{x-t}{y}}+e^{\frac{i\beta\pi}{2}}\sum_{n=0}^{\infty}\left(\frac{2n+1}{y}\right)^{-\beta}e^{-\pi i(2n+1)\frac{x-t}{y}}\right)dt\\
			=&\frac{1}{\pi^{\alpha+\beta}y^2}\left(e^{-\frac{\alpha+\beta}{2}i\pi}I_1+e^{\frac{\beta-\alpha}{2}i\pi}I_2+e^{\frac{\alpha-\beta}{2}i\pi}I_3+e^{\frac{\alpha+\beta}{2}i\pi}I_4\right),
		\end{align*}
		where
		\begin{align*}
			&I_1=\int_{0}^{y}\left(\sum_{m=0}^{\infty}\left(\frac{2m+1}{y}\right)^{-\alpha}e^{\pi i(2m+1)\frac{t}{y}}\right)\left(\sum_{n=0}^{\infty}\left(\frac{2n+1}{y}\right)^{-\beta}e^{\pi i(2n+1)\frac{x-t}{y}}\right)dt,\\
			&I_2=\int_{0}^{y}\left(\sum_{m=0}^{\infty}\left(\frac{2m+1}{y}\right)^{-\alpha}e^{\pi i(2m+1)\frac{t}{y}}\right)\left(\sum_{n=0}^{\infty}\left(\frac{2n+1}{y}\right)^{-\beta}e^{-\pi i(2n+1)\frac{x-t}{y}}\right)dt,\\
			&I_3=\int_{0}^{y}\left(\sum_{m=0}^{\infty}\left(\frac{2m+1}{y}\right)^{-\alpha}e^{-\pi i(2m+1)\frac{t}{y}}\right)\left(\sum_{n=0}^{\infty}\left(\frac{2n+1}{y}\right)^{-\beta}e^{\pi i(2n+1)\frac{x-t}{y}}\right)dt,\\
			&I_4=\int_{0}^{y}\left(\sum_{m=0}^{\infty}\left(\frac{2m+1}{y}\right)^{-\alpha}e^{-\pi i(2m+1)\frac{t}{y}}\right)\left(\sum_{n=0}^{\infty}\left(\frac{2n+1}{y}\right)^{-\beta}e^{-\pi i(2n+1)\frac{x-t}{y}}\right)dt.
		\end{align*}
		From the integral formula
		$$\int_{0}^{y}e^{2\pi ik\cdot\frac{t}{y}}dt=\begin{cases}
			0, & \textrm{if}~k\neq 0,\\
			y, & \textrm{if}~k=0,
		\end{cases}$$
		we obtain
		\begin{align*}
			I_1&=\sum_{m=0}^{\infty}\sum_{n=0}^{\infty}\left(\frac{2m+1}{y}\right)^{-\alpha}\left(\frac{2n+1}{y}\right)^{-\beta}e^{\pi i(2n+1)\frac{x}{y}}\int_{0}^{y}e^{2\pi i(m-n)\frac{t}{y}}dt\\
			&=y\sum_{m=0}^{\infty}\left(\frac{2m+1}{y}\right)^{-(\alpha+\beta)}e^{\pi i(2m+1)\frac{x}{y}},\\
			I_2&=\sum_{m=0}^{\infty}\sum_{n=0}^{\infty}\left(\frac{2m+1}{y}\right)^{-\alpha}\left(\frac{2n+1}{y}\right)^{-\beta}e^{-\pi i(2n+1)\frac{x}{y}}\int_{0}^{y}e^{2\pi i(n+m+1)\frac{t}{y}}dt=0,\\
			I_3&=\sum_{m=0}^{\infty}\sum_{n=0}^{\infty}\left(\frac{2m+1}{y}\right)^{-\alpha}\left(\frac{2n+1}{y}\right)^{-\beta}e^{\pi i(2n+1)\frac{x}{y}}\int_{0}^{y}e^{-2\pi i(n+m+1)\frac{t}{y}}dt=0,\\
			I_4&=\sum_{m=0}^{\infty}\sum_{n=0}^{\infty}\left(\frac{2m+1}{y}\right)^{-\alpha}\left(\frac{2n+1}{y}\right)^{-\beta}e^{-\pi i(2n+1)\frac{x}{y}}\int_{0}^{y}e^{2\pi i(n-m)\frac{t}{y}}dt\\
			&=y\sum_{m=0}^{\infty}\left(\frac{2m+1}{y}\right)^{-(\alpha+\beta)}e^{-\pi i(2m+1)\frac{x}{y}}.
		\end{align*}
		Therefore, by (\ref{(1-6-1)}) we have
		\begin{align*}
			f(x,y)=&\frac{1}{\pi^{\alpha+\beta}y}\cdot\left(e^{-\frac{\alpha+\beta}{2}i\pi}\sum_{m=0}^{\infty}\left(\frac{2m+1}{y}\right)^{-(\alpha+\beta)}e^{\pi i(2m+1)\frac{x}{y}}\right.\\
			&\left.+e^{\frac{\alpha+\beta}{2}i\pi}\sum_{m=0}^{\infty}\left(\frac{2m+1}{y}\right)^{-(\alpha+\beta)}e^{-\pi i(2m+1)\frac{x}{y}}\right)\\
			=\,&\frac{y^{\alpha+\beta-1}}{\Gamma(\alpha+\beta)}\zeta_E\left(1-\alpha-\beta,\frac{x}{y}\right).
		\end{align*}
		Thus for $0<x\leq y,$ we obtain
		\begin{equation}\label{(4-5)}
			\begin{aligned}
				f(x,y)=\frac{y^{\alpha+\beta-1}}{\Gamma(\alpha+\beta)}\zeta_E\left(1-\alpha-\beta,\frac{x}{y}\right).
			\end{aligned}
		\end{equation}
		Subtracting Eq. (\ref{(4-5)}) from  Eq. (\ref{(4-4)}) gives
		\begin{align*}
			f(x+y,y)=\frac{y^{\alpha+\beta-1}}{\Gamma(\alpha+\beta)}\zeta_E\left(1-\alpha-\beta,
			\frac{x+y}{y}\right),
		\end{align*}
		which shows that equation Eq. (\ref{(4-5)}) also holds when $x>y$.
		Finally, replacing $x$ with $x-y$ in Eq. (\ref{(4-4)}) gives
		\begin{align*}
			f(x-y,y)=\frac{y^{\alpha+\beta-1}}{\Gamma(\alpha+\beta)}\zeta_E\left(1-\alpha-\beta,\frac{x-y}{y}\right),
		\end{align*}
		which shows that Eq. (\ref{(4-5)}) also holds when $x\leq 0$.
		In conclusion, Eq. (\ref{(4-5)}) holds for all $x\in\mathbb{R}$, completing the proof.

	\section{Examples}
	In this section, we shall provide several examples of alternating invariant functions, including the Gamma function $\tilde{\Gamma}(x)$ corresponding to the alternating Hurwitz zeta function (Example \ref{Example 5.3}), several appropriate combinations of trigonometric, exponential functions and  logarithmic functions (Examples \ref{Example 5.4}--\ref{Example 5.9}), et al.

		\begin{Example}\label{Example 5.3}
		For $x \in \mathbb{R}$ and $y > 0$, let
		\begin{align*}
			f(x, y) = 
			\begin{cases} 
				\log \left| \frac{\tilde{\Gamma}\left( \frac{x}{y}\right) }{\sqrt{y}} \right|, & \text{if } \frac{x}{y} 
				\notin \mathbb{Z}_{\leq0}, \\
				\log \left( \frac{\left(-\frac{x}{y}-1\right)!!}{\left(-\frac{x}{y}\right)!!}  \sqrt{y}\cdot \tilde{\Gamma}(2) \right), & \text{if } \frac{x}{y} \in \{0, -2, -4, \ldots\}, \\
				\log \left( \frac{\left(-\frac{x}{y}-1\right)!!}{\left(-\frac{x}{y}\right)!!} \cdot  \frac{\tilde{\Gamma}(1)}{\sqrt{y^3}} \right), & \text{if } \frac{x}{y} \in \{-1, -3, -5, \ldots\}.
			\end{cases}
		\end{align*}
		Then $f(x,y) \in \mathrm{InV}^*$.
	\end{Example}
	To see this example, we need the following lemma.
	\begin{lemma}\textup{(Recurrence relation for $\tilde{\Gamma}(x)$, see \cite[Theorem 2.11]{47})}
		For $\textup{Re}(x) > 0$ and $n \in \mathbb{N},$ we have
		\begin{align*}
			\left(\tilde{\Gamma}(x + n)\right)^{(-1)^n} = \prod_{k=0}^{n-1} \left(\frac{2(x + k)}{\pi}\right)^{(-1)^k} \tilde{\Gamma}(x).
		\end{align*}
		In concrete, if $n$ is even, then
		\begin{equation}\label{(2-3)}
			\tilde{\Gamma}(x + n) = \frac{x + n - 2}{x + n - 1} \cdot \frac{x + n - 4}{x + n - 3} \cdots \frac{x}{x + 1} \cdot \tilde{\Gamma}(x);
		\end{equation}
		if $n$ is odd, then
		\begin{equation}\label{(2-4)}
			\tilde{\Gamma}(x + n) = \frac{x + n - 2}{x + n - 1} \cdot \frac{x + n - 4}{x + n - 3} \cdots \frac{x + 1}{x + 2} \cdot \tilde{\Gamma}(x + 1).
		\end{equation}
	\end{lemma}
	\begin{proof}[Proof of Example \ref{Example 5.3}]
		By Eq. (\ref{(1-9)}), when $z \notin \mathbb{Z}_{\leq0}$, the following distribution formula holds:
		\begin{align}
			\tilde{\Gamma}(nz) = \frac{1}{\sqrt{n}} \prod_{r=0}^{n-1} \tilde{\Gamma}\left(z + \frac{r}{n}\right)^{(-1)^r},
		\end{align}
		where $n$ is a positive odd integer.     
		Let $x = nz$. Then we have
		\begin{align*}
			\prod_{r=0}^{n-1}\left(  \frac{\tilde{\Gamma}\left(\frac{x + r}{n}\right) }{\sqrt{ny}}\right)  ^{(-1)^r}=\frac{\tilde{\Gamma}(x)}{\sqrt{y}}.
		\end{align*}
		Thus, for $\frac{x}{y} \notin \mathbb{Z}_{\leq0}$, we have
		\begin{align*}
			\sum\limits_{r=0}^{n-1}(-1)^r\log \left| \frac{\tilde{\Gamma}\left( \frac{x+ry}{ny}\right) }{\sqrt{ny}} \right|=\log \left|
			\prod_{r=0}^{n-1} \left(\frac{\tilde{\Gamma}\left(\frac{\frac{x}{y} + r}{n}\right) }{\sqrt{ny}}\right)^{(-1)^r}\right| =
			\log \left| \frac{\tilde{\Gamma}\left( \frac{x}{y}\right) }{\sqrt{y}} \right|.
		\end{align*}
		
		On the other hand, for $m = kn + l$ ($k \in \mathbb{N}_{0}$ and $0 \leq l \leq n-1$), we have
	\begin{align*}
		&\sum\limits_{r=0}^{n-1}(-1)^rf(x+ry,ny)\\
		=&	\lim_{\frac{x}{y} \to -m}\left( \sum\limits_{\substack{r=0 \\ r \neq l}}^{n-1}(-1)^r\log \left| \frac{\tilde{\Gamma}\left( \frac{x+ry}{ny}\right) }{\sqrt{ny}} \right|+(-1)^lf(x+ly,ny)\right) \\
		=&	\lim_{\frac{x}{y} \to -m}\sum\limits_{\substack{r=0 \\ r \neq l}}^{n-1}(-1)^r\log \left| \frac{\tilde{\Gamma}\left( \frac{x+ry}{ny}\right) }{\sqrt{ny}} \right|+(-1)^lf(-kny,ny)\\
		=&		\lim_{\frac{x}{y} \to -m}\left( \log \left| \frac{\tilde{\Gamma}\left( \frac{x}{y}\right) }{\sqrt{y}} \right|-
		(-1)^l\log \left| \frac{\tilde{\Gamma}\left( \frac{x+ly}{ny}\right) }{\sqrt{ny}}\right|\right) +(-1)^lf(-kny,ny)	\\
		=&	\lim_{\frac{x}{y} \to -m}\log \left|\frac{\tilde{\Gamma}\left( \frac{x}{y}\right) \cdot\frac{1}{\sqrt{y}}}{	\left( \tilde{\Gamma}\left( \frac{\frac{x}{y}+l}{n}\right) \cdot\frac{1}{\sqrt{ny}}\right) ^{(-1)^l}}\right|+(-1)^lf(-kny,ny).
	\end{align*}
		There are two main cases.
		First, if $m$ and $k$ are both even, then $l$ is even. By Eq. (\ref{(2-3)}) we have
		\begin{align*}
			&\sum\limits_{r=0}^{n-1}(-1)^rf(x+ry,ny)\\
			=&\lim_{\frac{x}{y}  \to -m}\log \left|\frac{\tilde{\Gamma}\left( \frac{x}{y}\right) \cdot\frac{1}{\sqrt{y}}}{	\left( \tilde{\Gamma}\left( \frac{\frac{x}{y}+l}{n}\right) \cdot\frac{1}{\sqrt{ny}}\right) }\right|+	\log \left( \frac{(k-1)!!}{k!!}  \sqrt{ny}\cdot \tilde{\Gamma}(2) \right)\\
			=&\lim_{\frac{x}{y}  \to -m}\log \left|	\frac{	\frac{ \prod\limits_{\substack{r=1 \\ 2 \nmid r}}^{m+1} \left( \frac{x}{y} + r \right) }{ \prod\limits_{\substack{r=0 \\ 2 \mid r}}^{m} \left( \frac{x}{y} + r \right) }
				\cdot\tilde{\Gamma}\left( \frac{x}{y}+m+2\right) \cdot\frac{1}{\sqrt{y}}}{
				\frac{\prod\limits_{\substack{r=1 \\ 2 \nmid r}}^{k+1} \left( \frac{\frac{x}{y} + l}{n} + r \right)}{\prod\limits_{\substack{r=0 \\ 2 \mid r}}^{k} \left( \frac{\frac{x}{y} + l}{n} + r \right)}\cdot\tilde{\Gamma}\left( \frac{\frac{x}{y}+l}{n}+k+2\right) \cdot\frac{1}{\sqrt{ny}}}\right|+	\log \left( \frac{(k-1)!!}{k!!}  \sqrt{ny}\cdot \tilde{\Gamma}(2) \right) \\
			=&\log \left| \frac{\frac{(m-1)!!}{m!!}\frac{1}{\sqrt{y}}\cdot \tilde{\Gamma}(2)}{n\cdot\frac{(k-1)!!}{k!!}\frac{1}{\sqrt{ny}}\cdot \tilde{\Gamma}(2)}
			\right|+	\log \left( \frac{(k-1)!!}{k!!}  \sqrt{ny}\cdot \tilde{\Gamma}(2) \right)\\		
			=&
			\log \left| \frac{\frac{(m-1)!!}{m!!}\sqrt{y}\cdot \tilde{\Gamma}(2)}{\frac{(k-1)!!}{k!!}\sqrt{ny}\cdot \tilde{\Gamma}(2)}
			\right|+	\log \left( \frac{(k-1)!!}{k!!}  \sqrt{ny}\cdot \tilde{\Gamma}(2) \right)\\
			=&\log \left( \frac{(m-1)!!}{m!!}  \sqrt{y}\cdot \tilde{\Gamma}(2) \right)\\
			=&f(x,y),
		\end{align*}
		here we have used 
		\begin{align*}
		\lim_{\frac{x}{y}  \to -m}\left|\frac{\frac{\frac xy +l}{n}+k}{\frac xy +m}\right|=\frac1n.	
		\end{align*}

 For the case where at least one of \(m\) and \(k\) is odd, this example can be verified using a similar argument.
									\end{proof}
	
	\begin{Example}\label{Example 5.4}
		For $y>0,$ define
		\begin{align*}
			f(x, y) = 
			\begin{cases} 
				\log \left| \tan \frac{\pi x}{2y} \right|, & \text{if } \frac{x}{y} \notin \mathbb{Z}, \\
				-(-1)^{\frac{x}{y}} \log y, & \text{if } \frac{x}{y} \in \mathbb{Z}.
			\end{cases}
		\end{align*}
		Then $f(x,y)\in \mathrm{InV}^*$.
	\end{Example}
	
	\begin{proof}
When \(\frac{x}{y} \notin \mathbb{Z}\), we obtain
\begin{align*}
\log \left| \tan \frac{\pi x}{2y} \right| = \frac{-2}{y} \sum_{m=0}^{\infty}
\frac{\cos\left( (2m+1) \cdot \frac{\pi x}{y} \right)}{\frac{2m+1}{y}}.
\end{align*}
Since \(\frac{x}{y} \notin \mathbb{Z}\), for any \(r \in \{0,1,\ldots,n-1\}\) 
we have \(\frac{x+ry}{ny} \notin \mathbb{Z}\). Thus by Proposition \ref{Pro 3.7}, we get
\begin{align*}
\sum_{r=0}^{n-1} (-1)^r f(x+ry,ny) = f(x,y)
\end{align*}
for odd \(n\) and \(y > 0\).
		
If $\frac{x}{y}\in \mathbb{Z}$, there exists a unique $s\in \{0,1,\ldots,n-1\}$ such that
$\frac{x+sy}{ny}\in \mathbb{Z}$. Indeed, assuming $\frac{x}{y} = kn+t$ where $1 \le t \le n$, 
we find that $s=n-t$. Thus
\begin{align*}
\sum_{\substack{r=0 \\ r \neq s}}^{n-1}(-1)^r\log \left| \tan\pi\cdot \frac{x+ry}{2ny} \right|
&= \sum_{r=0}^{s-1}(-1)^r\log \left| \tan\left( \pi \cdot\frac{\frac{x}{y}+r}{2n}\right) \right| \\
&\quad + \sum_{r=s+1}^{n-1}(-1)^r\log \left| \tan\left( \pi \cdot\frac{\frac{x}{y}+r}{2n}\right) \right|.
\end{align*}
Using the identity $\tan\left(x+\frac{k\pi}{2}\right) = (-1)^k \tan(x)^{(-1)^k}$, this becomes 
\begin{align*}
\sum_{\substack{r=0 \\ r \neq s}}^{n-1}(-1)^r\log \left| \tan\left(\pi\cdot \frac{x+ry}{2ny}\right) 
\right| 
&=\sum_{r=0}^{s-1}(-1)^{r+k}\log \left| \tan\left(\pi \cdot\frac{t+r}{2n}\right) \right| \\
&\quad + \sum_{r=s+1}^{n-1}(-1)^{r+k+1}\log \left| \tan\left(\pi \cdot\frac{t+r-n}{2n}\right) \right|.
\end{align*}
Since $n$ is odd, we have
\begin{align*}
(-1)^r(-1)^k=(-1)^{r+t}(-1)^{kn+t} \text{ and } (-1)^r(-1)^{k+1}=(-1)^{r+t-n}(-1)^{kn+t}.
\end{align*}
Through variable substitution, we obtain
\begin{align*}
&\quad\sum_{\substack{r=0 \\ r \neq s}}^{n-1}(-1)^r\log \left| \tan\left( \pi\cdot \frac{x+ry}{2ny}\right) \right| \\
&= \sum_{r=0}^{s-1}(-1)^{r+t}(-1)^{kn+t}\log \left| \tan\frac{\pi (t+r)}{2n} \right| \\
&\quad + \sum_{r=s+1}^{n-1}(-1)^{r+t-n}(-1)^{kn+t}\log \left| \tan \frac{\pi(t+r-n)}{2n} \right| \\
&= (-1)^{\frac{x}{y}}\cdot\left( \sum_{r=t}^{n-1}(-1)^r\log \left| \tan \frac{\pi r}{2n} \right| + \sum_{r=1}^{t-1}(-1)^r\log \left| \tan\frac{\pi r}{2n} \right|\right) \\
&= (-1)^{\frac{x}{y}}\cdot \log \prod_{r=1}^{n-1}\left| \tan \frac{\pi r}{2n} \right|^{(-1)^r}.
\end{align*}
Evaluating the product:
\begin{align*}
\prod\limits_{r=1}^{n-1}\left( \tan \frac{\pi r}{2n} \right)^{(-1)^r}
=&\prod\limits_{r=1}^{n-1}\left( \frac{1-e^{\frac{\pi ri}{n}}}{1+e^{\frac{\pi ri}{n}}} \right)^{(-1)^r}\\
=&\frac{\prod\limits_{\substack{r=1 \\ 2 \mid r}}^{n-1}\left( 1-e^{\frac{\pi ri}{n}}\right)\cdot \prod\limits_{\substack{r=1 \\ 2 \nmid r}}^{n-1}\left( 1+e^{\frac{\pi ri}{n}}\right) }{\prod\limits_{\substack{r=1 \\ 2 \nmid r}}^{n-1}\left( 1-e^{\frac{\pi ri}{n}}\right)\cdot \prod\limits_{\substack{r=1 \\ 2 \mid r}}^{n-1}\left( 1+e^{\frac{\pi ri}{n}}\right) }\\
=&\frac{\prod\limits_{\substack{r=1 \\ 2 \mid r}}^{n-1}\left( 1-e^{\frac{\pi ri}{n}}\right)\cdot\prod\limits_{\substack{s=n+1 \\ 2 \mid s}}^{2n-1}\left( 1-e^{\frac{\pi si}{n}}\right)
			}{\prod\limits_{\substack{r=1 \\ 2 \nmid r}}^{n-1}\left( 1-e^{\frac{\pi ri}{n}}\right)\cdot\prod\limits_{\substack{s=n+1 \\ 2 \nmid s}}^{2n-1}\left( 1-e^{\frac{\pi si}{n}}\right) }\\
			&\quad(\text{by letting}~s=r+n)\\       
						=&\frac{\prod\limits_{r=1}^{n-1}\left( 1-e^{\frac{2\pi ri}{n}}\right)}{\prod\limits_{\substack{r=1 \\ r \neq \frac{n+1}{2}}}^{n}\left( 1-e^{\frac{(2r-1)\pi i}{n}}\right)}\\
=&\lim_{x  \to 1}\frac{\frac{x^n-1}{x-1}}{\frac{x^n+1}{x+1}}=n,
\end{align*}
we get
\begin{align*}
\sum\limits_{r=0}^{n-1}(-1)^rf(x+ry,ny)
&=\sum_{\substack{r=0 \\ r \neq s}}^{n-1}(-1)^r\log \left| \tan\left(\pi\cdot \frac{x+ry}{2ny}\right) \right| +(-1)^{\frac{x}{y}}\log n\\
&= -(-1)^s(-1)^{\frac{x+sy}{ny}}\log ny + (-1)^{\frac{x}{y}}\log n.
\end{align*}
Since $n$ is odd, this simplifies to
$$\sum\limits_{r=0}^{n-1}(-1)^rf(x+ry,ny)
 = -(-1)^{\frac{x}{y}}\log ny + (-1)^{\frac{x}{y}}\log n
 =f(x,y).$$
In conclusion, we have $f(x,y)\in \mathrm{InV}^*$.
\end{proof}

	\begin{Example}\label{Example 5.5}
		For $y > 0$, let
		\begin{align*}
			f(x, y) = 
			\begin{cases}
				\frac{1}{y} \csc \pi \frac{x}{y}, & \text{if } \frac{x}{y} 
				\notin \mathbb{Z}, \\
				0, & \text{if } \frac{x}{y} \in \mathbb{Z}.
			\end{cases}
		\end{align*}
		Then $f(x, y) \in \mathrm{InV}^*$.
	\end{Example}
	\begin{proof}
		Since $\frac{\partial}{\partial x}\log|\tan\frac{\pi x}{2y}|=\frac{\pi}{y}\csc\frac{\pi x}{y}$, Example \ref{Example 5.5} follows from Example  \ref{Example 5.4} and Proposition \ref{Pro 3.2} immediately.
	\end{proof}

	\begin{Example}\label{Example 5.6}
		For $a \in \mathbb{R}$, let
		\begin{align*}
			f(x, y) = 
			\begin{cases}
				(-1)^{\frac{a - x}{y}}, &\textrm{if}~ \frac{a - x}{y} \in \mathbb{Z}, \\
				0, & \textrm{if}~\frac{a - x}{y} 
				\notin \mathbb{Z}.
			\end{cases}
		\end{align*}
		Then $f(x, y) \in \mathrm{InV}^*$.
	\end{Example}
	
	\begin{proof}
		If $\frac{a - x}{y} \in \mathbb{Z}$, then there exists a unique $s \in \{0, 1, \ldots, n-1\}$ such that $\frac{a - x}{y} \equiv s \pmod{n}$. Assume $\frac{a - x}{y} = kn + s$ for some $k\in\mathbb{N}_{0}$. 
		Since $n$ is odd, we have
		\begin{align*}
			\sum_{r=0}^{n-1} (-1)^r f(x + ry, ny) = (-1)^s (-1)^{\frac{a - x - sy}{ny}} = (-1)^s (-1)^k.
		\end{align*}
		Also due to $n$ being odd, $(-1)^s (-1)^k = (-1)^s (-1)^{kn} = (-1)^{kn + s}$. Thus,
		\begin{align*}
			\sum_{r=0}^{n-1} (-1)^r f(x + ry, ny) = (-1)^{kn + s} = (-1)^{\frac{a - x}{y}} = f(x, y).
		\end{align*}
		If $\frac{a - x}{y} \notin \mathbb{Z}$, then for any $r \notin \{0, 1, \ldots, n-1\}$, we have $\frac{a - x - ry}{ny} \notin \mathbb{Z}$. In this case,
		\begin{align*}
			\sum_{r=0}^{n-1} (-1)^r f(x + ry, ny) = 0 = f(x, y).
		\end{align*}
		Therefore, $f(x, y) \in \mathrm{InV}^*$.
	\end{proof}

	\begin{Example}\label{Example 5.7}
		For $a > 0$ and $a \neq 1$, we have $\frac{a^x}{a^y + 1} \in \mathrm{InV}^*$.
	\end{Example}
	\begin{proof}
		From the following equation
		\begin{align*}
			\sum_{r=0}^{n-1} (-1)^r \frac{a^{x + ry}}{a^{ny} + 1} = \frac{a^x}{a^{ny} + 1} \sum_{r=0}^{n-1} (-1)^r (-a^y)^r = \frac{a^x}{a^{ny} + 1} \cdot \frac{1 - (-a^y)^n}{1 + a^y} = \frac{a^x}{a^y + 1},
		\end{align*}
		the conclusion holds.
	\end{proof}
	
	\begin{Example}\label{Example 5.8}
		For $r > 0$ and $r \neq 1$, let
		\begin{align*}
			f(x, y) = \frac{r^{x + y} \cos(x - y)\theta + r^x \cos x\theta}{1 + 2r^y \cos y\theta + r^{2y}}, \quad g(x, y) = \frac{r^{x + y} \sin(x - y)\theta + r^x \sin x\theta}{1 + 2r^y \cos y\theta + r^{2y}}.
		\end{align*}
		Then we have $f(x, y)$ and $g(x, y) \in \mathrm{InV}^*$.
	\end{Example}
	
	\begin{proof}
		By Euler's identity $e^{i\theta} = \cos\theta + i\sin\theta$, we have
		\begin{align*}
			\frac{(re^{i\theta})^x}{(re^{i\theta})^y + 1} = \frac{r^x (\cos x\theta + i\sin x\theta)}{r^y \cos y\theta + 1 + ir^y \sin y\theta}.
		\end{align*}
		In the right-hand side, multiplying the numerator and denominator by the conjugate of the denominator, we get
		\begin{align*}
			\frac{(re^{i\theta})^x}{(re^{i\theta})^y + 1} = \frac{r^x (\cos x\theta + i\sin x\theta)(r^y \cos y\theta + 1 - ir^y \sin y\theta)}{(r^y \cos y\theta + 1)^2 + (r^y \sin y\theta)^2}.
		\end{align*}
		Then expanding the numerator and denominator, we have
		\begin{align*}
			\frac{(re^{i\theta})^x}{(re^{i\theta})^y + 1} = \frac{r^{x + y} \cos(x - y)\theta + r^x \cos x\theta}{1 + 2r^y \cos y\theta + r^{2y}} + i \cdot \frac{r^{x + y} \sin(x - y)\theta + r^x \sin x\theta}{1 + 2r^y \cos y\theta + r^{2y}}.
		\end{align*}
		Since
		\begin{align*}
			f(x, y) + ig(x, y) = \frac{(re^{i\theta})^x}{(re^{i\theta})^y + 1},
		\end{align*}
		by Proposition \ref{Pro 3.3} and Example \ref{Example 5.7} we conclude that $f(x, y)$ and $g(x, y) \in \mathrm{InV}^*$.
	\end{proof}
	
	\begin{Example}\label{Example 5.9}
		For $0 < r < 1$, let
		\begin{align*}
			f(x, y) &= \frac{\cos\left(\pi\frac{x}{y}\right)\left(1 - r^{\frac{2}{y}}\right)r^{\frac{1}{y}}}{y\left(1 - 2r^{\frac{2}{y}}\cos\left(2\pi\frac{x}{y}\right) + r^{\frac{4}{y}}\right)}, \\
			g(x, y) &= \frac{\sin\left(\pi\frac{x}{y}\right)\left(1 + r^{\frac{2}{y}}\right)r^{\frac{1}{y}}}{y\left(1 - 2r^{\frac{2}{y}}\cos\left(2\pi\frac{x}{y}\right) + r^{\frac{4}{y}}\right)}.
		\end{align*}
		Then we have $f(x, y)$ and $g(x, y) \in \mathrm{InV}^*$.
	\end{Example}
	
	\begin{proof}
		For $0 < x < 1$, we have
		\begin{align*}
			\frac{x}{1 - x^2} = \sum_{m=0}^{\infty} x^{2m + 1}.
		\end{align*}
		Thus,
		\begin{align*}
			F(x, y) &= \frac{r^{\frac{1}{y}} e^{\pi i \frac{x}{y}}}{y \left(1 - r^{\frac{2}{y}} e^{2\pi i \frac{x}{y}}\right)} \\
			&= \frac{1}{y} \sum_{m=0}^{\infty} r^{\frac{2m + 1}{y}} e^{(2m + 1)\pi i \frac{x}{y}}.
		\end{align*}
		From the proof of Proposition \ref{Pro 3.7}, we know that $F(x, y) \in \mathrm{InV}^*$.
		
		In addition, by Euler's identity $e^{i\theta} = \cos\theta + i\sin\theta$, we have
		\begin{align*}
			F(x, y) &= \frac{r^{\frac{1}{y}} \left(\cos\left(\pi\frac{x}{y}\right) + i\sin\left(\pi\frac{x}{y}\right)\right)}{y \left(1 - r^{\frac{2}{y}} \cos\left(2\pi\frac{x}{y}\right) - i r^{\frac{2}{y}} \sin\left(2\pi\frac{x}{y}\right)\right)} \\
			&= \frac{r^{\frac{1}{y}} \left(\cos\left(\pi\frac{x}{y}\right) + i\sin\left(\pi\frac{x}{y}\right)\right) \left(1 - r^{\frac{2}{y}} \cos\left(2\pi\frac{x}{y}\right) + i r^{\frac{2}{y}} \sin\left(2\pi\frac{x}{y}\right)\right)}{y \left(1 - 2r^{\frac{2}{y}} \cos\left(2\pi\frac{x}{y}\right) + r^{\frac{4}{y}}\right)} \\
			&= \frac{\cos\left(\pi\frac{x}{y}\right) \left(1 - r^{\frac{2}{y}}\right) r^{\frac{1}{y}}}{y \left(1 - 2r^{\frac{2}{y}} \cos\left(2\pi\frac{x}{y}\right) + r^{\frac{4}{y}}\right)} + i \cdot \frac{\sin\left(\pi\frac{x}{y}\right) \left(1 + r^{\frac{2}{y}}\right) r^{\frac{1}{y}}}{y \left(1 - 2r^{\frac{2}{y}} \cos\left(2\pi\frac{x}{y}\right) + r^{\frac{4}{y}}\right)} \\
			&= f(x, y) + i g(x, y).
		\end{align*}
		Therefore, by Proposition \ref{Pro 3.3}, we have $f(x, y)$ and $g(x, y) \in \mathrm{InV}^*$.
	\end{proof}
	
\section*{Acknowledgements} 
 The authors are enormously grateful to the anonymous referee for his/her careful
reading of this paper, and for his/her  valuable comments and suggestions. 

Su Hu is supported by the Natural Science Foundation of Guangdong Province, China (No. 2024A1515012337).  Min-Soo Kim is supported by the National Research Foundation of Korea(NRF) grant funded by the Korea government(MSIT) (No. NRF-2022R1F1A1065551). 

\section*{Statements and Declarations}
\subsection*{Declaration of competing interest}
The authors declare that they have no known competing financial interests or personal relationships that could have appeared to influence the work reported in this paper.

Authors declare that they do not have any conflict of interest.

\subsection*{Data availability}
No data was used for the research described in the article.

	\bibliography{central}

\begin{thebibliography}{00}
		
		\bibitem{27} M. Abramowitz and I. Stegun (eds.),
		\textit{Handbook of mathematical functions with formulas, graphs and mathematical tables}, Dover, New York, 1972.
		
		
		
		
		
		\bibitem{14} J. Choi and H.M. Srivastava, \textit{The multiple Hurwitz zeta function and the multiple Hurwitz Euler eta function,}  Taiwanese J. Math. 2011, \textbf{15} (2), 501--522.
		
		\bibitem{19} H. Cohen,
		\textit{Number Theory Vol. II: Analytic and Modern Tools},
		Graduate Texts in Mathematics, 240, Springer, New York, 2007.
		
		\bibitem{22} D. Cvijovi\'c,
		\textit{A note on convexity properties of functions related to the Hurwitz zeta and alternating Hurwitz zeta function},
		J. Math. Anal. Appl. \textbf{487} (2020), no. 1, 123972, 8 pp.
		
		
		
		
		\bibitem{15} S. Hu, D. Kim and M.-S. Kim,
		\textit{Special values and integral representations for the Hurwitz-type Euler zeta functions},
		J. Korean Math. Soc. \textbf{55} (2018), no. 1, 185--210.
		
		
		\bibitem{23} S. Hu and M.-S. Kim,
		\textit{On Dirichlet's lambda function},
		J. Math. Anal. Appl. \textbf{478} (2019), no. 2, 952--972.
		
		\bibitem{16} S. Hu and M.-S. Kim,
		\textit{On the Stieltjes constants and gamma functions with respect to alternating Hurwitz zeta functions},
		J. Math. Anal. Appl. \textbf{509} (2022), no. 1, Paper No. 125930, 33 pp.
		
		\bibitem{17} S. Hu and M.-S. Kim,
		\textit{Asymptotic expansions for the alternating Hurwitz zeta function and its derivatives},
		J. Math. Anal. Appl. \textbf{537} (2024), no. 1, Paper No. 128306, 29 pp.
		
		
		\bibitem{46}  M.-S. Kim and S. Hu,
		\textit{On $p$-adic Hurwitz-type Euler zeta functions},
		J. Number Theory \textbf{132} (2012), 2977--3015.
		
		\bibitem{44} M.-S. Kim and S. Hu, \textit{On $p$-adic Diamond-Euler Log Gamma functions},
		J. Number Theory \textbf{133} (2013), 4233--4250.
		
		\bibitem{MOS} W. Magnus, F. Oberhettinger and R.P. Soni, 
\textit{Formulas and theorems for the special functions of mathematical physics.}
Third enlarged edition.
Die Grundlehren der mathematischen Wissenschaften, Band 52.
Springer-Verlag New York, Inc., New York, 1966.
		
		
		
		
		\bibitem{10} H.M. Srivastava, B. Kurt and  Y. Simsek, \textit{Some families of Genocchi type polynomials and their interpolation functions,} Integral Transforms Spec. Funct. 2012, \textbf{23} (12), 919--938.
		
		
		
		
		
		\bibitem{33} Z.-H. Sun, \textit{On the properties of invariant functions,}  Bull. Sci. Math. 2023, \textbf{189}, Paper No. 103347, 24 pp.
		
		\bibitem{2}  Z.-W. Sun,
		\textit{Introduction to Bernoulli and Euler polynomials},
		a lecture given in Taiwan on June 6, 2002, \url{http://maths.nju.edu.cn/~zwsun/BerE.pdf}.
		
		
		\bibitem{30} Z.-W. Sun, \textit{Systems of congruences with multipliers},  Nanjing Univ. J. -Math. Biq., 1989 \textbf{6}, 124--133.
		
		\bibitem{31} Z.-W. Sun, \textit{Algebraic approaches to periodic arithmetical maps},  J. Algebra, 2001, \textbf{240} (2): 723--743.
		
		
		
		\bibitem{32} Z.-W. Sun, \textit{On covering equivalence.} Dordrecht: Kluwer Acad. Publ., 2002: 277–302.
		
		
		
		
		\bibitem{43} Z.X. Wang and D.R. Guo,
		\textit{Special functions},
		Translated from the Chinese by Guo and X.J. Xia,
		World Scientific Publishing Co., Inc., Teaneck, NJ, 1989.
		
		\bibitem{47} W. Wang, S. Hu and M.-S. Kim, \textit{On gamma functions with respect to the alternating Hurwitz zeta functions,}  preprint 2024, \url{https://arxiv.org/abs/2405.02854}.
		
		\bibitem{20} K.S. Williams and N.Y. Zhang,
		\textit{Special values of the Lerch zeta function and the evaluation of certain integrals},
		Proc. Amer. Math. Soc. \textbf{119} (1993), no. 1, 35--49.
		
		
		
	\end{thebibliography}

\end{document}